\documentclass{amsart}
\usepackage{amssymb}
\usepackage{pb-diagram}
\usepackage{color}
\usepackage[normalem]{ulem}

\newtheorem{theorem}{Theorem}[section]
\newtheorem{thm}[theorem]{Theorem}
\newtheorem{prop}[theorem]{Proposition}
\newtheorem{lem}[theorem]{Lemma}

\newtheorem{fact}[theorem]{Fact}
\newtheorem{cor}[theorem]{Corollary}

\newtheorem{lemma}[theorem]{Lemma}

\newtheorem{conj}[theorem]{Conjecture}
\newtheorem{question}[theorem]{Question}

\theoremstyle{definition}
\newtheorem{defn}[theorem]{Definition}
\newtheorem{definition}[theorem]{Definition}
\newtheorem{example}[theorem]{Example}

\theoremstyle{remark}

\newtheorem{remark}[theorem]{Remark}

\newcommand{\la}{\langle}
\newcommand{\ra}{\rangle}

\newcommand{\CH}{{\cal H}}
\newcommand{\sub}{\subseteq}

\newcommand{\WM}{\widetilde{\cal M}}

\newcommand{\cldim}{\ensuremath{\textup{scl-dim}}}
\newcommand{\dcl}{\operatorname{dcl}}
\newcommand{\scl}{\operatorname{scl}}

\newcommand{\rk}{\ensuremath{\textup{rk}}}
\newcommand{\rank}{\ensuremath{\textup{rank}}}
\newcommand{\ldim}{\ensuremath{\textup{ldim}}}

\newcommand{\bb}[1]{\ensuremath{\mathbb{#1}}}

\newcommand{\cal}[1]{\ensuremath{\mathcal{#1}}}
\newcommand{\Cal}[1]{\ensuremath{\mathcal{#1}}}

\newcommand{\Lrarr}{\ensuremath{\Leftrightarrow}}
\newcommand{\Rarr}{\ensuremath{\Rightarrow}}

\newcommand{\rarr}{\ensuremath{\rightarrow}}

\newcommand{\res}{\ensuremath{\upharpoonright}}

\newcommand{\sm}{\setminus}

\newcommand{\Z}{\mathbb{Z}}
\newcommand{\N}{\mathbb{N}}
\newcommand{\Q}{\mathbb{Q}}
\newcommand{\R}{\mathbb{R}}

\title[Tame expansions of o-minimal structures]
{Structure theorems in tame expansions of o-minimal structures by a dense set}

\begin{document}

\author {Pantelis  E. Eleftheriou}

\address{Department of Mathematics and Statistics, University of Konstanz, Box 216, 78457 Konstanz, Germany}

\email{panteleimon.eleftheriou@uni-konstanz.de}

\thanks{The first author was supported by an Independent Research Grant from the German Research Foundation (DFG) and a Zukunftskolleg Research Fellowship. The second author was partially supported by TUBITAK Career Grant 113F119. The third author was partially supported by NSF grant DMS-1300402.}

\author {Ayhan G\"{u}naydin}

\address{Department of Mathematics, Bo\u{g}azi\c{c}i University, Bebek, Istanbul, Turkey}

\email{ayhan.gunaydin@boun.edu.tr}


\author{Philipp Hieronymi}

\address{Department of Mathematics, University of Illinois at Urbana-Champaign, 1409 West Green Street, Urbana, IL 61801, USA}

\email{phierony@illinois.edu}


\date{\today}
\begin{abstract} We study sets and groups definable in tame expansions of o-minimal structures.
Let $\cal {\widetilde M}= \la \cal M, P\ra$ be an expansion of an o-minimal $\cal L$-structure $\cal M$ by a dense set $P$.  
We impose three tameness conditions on $\cal {\widetilde M}$ and prove a structure theorem for definable sets and functions in analogy with the cell decomposition theorem known for o-minimal structures. The structure theorem advances the state-of-the-art in all known examples of such $\WM$, as it achieves a decomposition of definable sets into \emph{unions} of `cones', instead of only boolean combinations of them. The proofs  involve induction on the notion of `large dimension' for definable sets, an invariant which we herewith introduce and analyze. 
Applications of the cone decomposition theorem include: (i) the large dimension of a definable set coincides with a suitable pregeometric dimension,
  and it is invariant under definable bijections,  (ii) every definable map is given by an $\cal L$-definable map off a subset of the domain of smaller large dimension, and (iii) around generic elements of a definable group, the group operation is given by an $\cal L$-definable map.

\end{abstract}

\begin{abstract} 
  We study sets and groups definable in tame expansions of o-minimal structures.
Let $\cal {\widetilde M}= \la \cal M, P\ra$ be an expansion of an o-minimal $\cal L$-structure $\cal M$ by a dense set $P$.
We impose three tameness conditions on $\cal {\widetilde M}$ and prove a structure theorem for definable sets and functions in analogy with the cell decomposition theorem known for o-minimal structures. Furthermore, we introduce a structural dimension for every definable set, generalizing the usual topological dimension. The structure theorem advances the state-of-the-art in all known examples of $\WM$, as it achieves a decomposition of definable sets into \emph{unions} of `cones', instead of only boolean combinations of them.
Applications include:
(i) the structural dimension of a definable set coincides with a suitable pregeometric dimension,  and it is invariant under definable bijections,
  (ii) every definable map is given by an $\cal L$-definable map off a subset of the domain of smaller dimension, and (iii) around generic elements of a definable group, the group operation is given by an $\cal L$-definable map.
\end{abstract}

\begin{abstract}
  We study sets and groups definable in tame expansions of o-minimal structures.
Let $\cal {\widetilde M}= \la \cal M, P\ra$ be an expansion of an o-minimal $\cal L$-structure $\cal M$ by a dense set $P$, such that three tameness conditions hold. We prove a structure theorem for definable sets and functions in analogy with the influential cell decomposition theorem known for o-minimal structures.  The structure theorem advances the state-of-the-art in all known examples of $\WM$, as it achieves a decomposition of definable sets into \emph{unions} of `cones', instead of only boolean combinations of them. We also develop the right dimension theory in the tame setting.
Applications include:
(i) the dimension of a definable set coincides with a suitable pregeometric dimension,  and it is invariant under definable bijections,
  (ii) every definable map is given by an $\cal L$-definable map off a subset of its domain of smaller dimension, and (iii) around generic elements of a definable group, the group operation is given by an $\cal L$-definable map.
\end{abstract}
 \maketitle
\section{Introduction}

Definable groups in models of first-order theories have been at the core of model theory for at least a period of three decades (see, for example, \cite{bn, ot, poi}) and have been crucially used in important applications of model theory to other areas of mathematics (such as in \cite{hr}). An indispensable tool in their analysis has  been a structure theorem for  the definable sets and types: analyzability of types and the existence of a rank in the  stable category, and a cell decomposition theorem and the associated topological dimension in the o-minimal setting. In this paper we establish a structure theorem for definable sets and functions in tame expansions of o-minimal structures, introduce and analyze the relevant notion of dimension and establish a local theorem for definable groups in this setting. Our structure theorem is inspired by a cone decomposition theorem known for semi-bounded o-minimal structures (\cite{ed-str, el-sbdgroups, pet-str}), which was also vitally   used in the analysis of definable groups therein (\cite{ep-sel2}).
The structure theorem has opened the way to other applications of the tame setting, beyond the study of definable groups, such as  the point counting theorems  in \cite{el-pw}.








Let us  briefly discuss the tame setting.
O-minimal structures were introduced and first studied by van den Dries \cite{vdd-tarski} and Knight-Pillay-Steinhorn \cite{kps, ps} and have since provided a rigid framework to study real algebraic and analytic geometry. They have enjoyed a wide spectrum of applications reaching out even to number theory and Diophantine geometry (such as in Pila's solution of certain cases of the Andr\'e-Oort  Conjecture \cite{pila}). However, o-minimality can only be used to model phenomena that are at least locally finite, or more precisely, objects that have only finitely many connected components. Tame expansions of o-minimal structures can further model phenomena that escape from the o-minimal context, but yet exhibit tame geometric behavior. They have recently seen significant growth (\cite{bz, beg, bh,dms1,vdd-dense,dg,gh, ms}) and are by now divided into two important classes of structures: those where every open definable set is already definable in the o-minimal reduct and those where an infinite discrete set is definable. We establish our cone decomposition theorem in the former category. In the second category,  a relevant structure theorem has already been obtained in \cite{tycho}, benefiting largely by the presence of definable choice in that setting (absent here).

We now fix our setting and describe the results of this paper.
Let $\cal M$ be an o-minimal expansion of an ordered group with underlying language $\cal L$. Let $\widetilde{\cal M}=\la \cal M, P\ra$ be an expansion of \cal M by a dense set $P$ so that certain tameness conditions hold (those are listed in Section \ref{sec-assumptions}). For example, $\widetilde{\cal M}$ can be a dense pair (\cite{vdd-dense}), or $P$ can be an independent set (\cite{dms2}) or a multiplicative group with the Mann Property (\cite{dg}).  To establish our structure theorem below, we introduce a new invariant for definable sets, the `large dimension', which turns out to coincide with the combinatorial dimension coming from a pregeometry in \cite{beg}. These results are in the  spirit of some standard and recent  literature. In an o-minimal structure, the cell decomposition theorem (\cite{vdd-book, kps}) is used to show that the associated `topological dimension' equals the combinatorial dimension coming from the $\dcl$-pregeometry (\cite{pi-eq}). In a semi-bounded structure, the cone decomposition theorem (\cite{ed-str, el-sbdgroups, pet-str}) is used to show that the associated `long dimension' equals the dimension coming from the short closure pregeometry (\cite{el-sbdgroups}). In both settings, the equivalence of the two dimensions has proven extremely powerful in many occasions and in particular in the analysis of definable groups (see, for example, \cite{el-sbdgroups, ep-sel2, elst, pi-groups}). Here, we apply the strategy from the semi-bounded setting to that of tame expansions of o-minimal structures and establish the analogous results in $\widetilde{\cal M}$.



In Sections 2 and 3 we include some  preliminaries and do preparatory work for  what follows. In  Section \ref{sec-cones}, we introduce the notions of a \emph{cone} and \emph{large dimension}. Although the definitions appear to be rather technical, we show in subsequent work that they are in fact optimal (see Section \ref{sec-optimality}, Question \ref{productcones} and \cite{el-opt}). In Section \ref{sec-str-thms}, we prove the following theorem.\\

 \noindent\textbf{Structure Theorem (\ref{str-thm-sets}).}
\begin{enumerate}
  \item \emph{Let $X\sub M^n$ be an $A$-definable set. Then $X$ is a finite union of $A$-definable cones.}
  \item \emph{Let $f:X\to M$ be an $A$-definable function. Then there is a finite collection $\mathcal{C}$ of   $A$-definable cones, whose union is $X$ and such that $f$ is fiber $\cal L_A$-definable with respect to each cone in $\cal{C}$.}
\end{enumerate}

\noindent We then conclude that the large dimension is invariant under definable bijections (Corollary \ref{bijection}). The above Structure Theorem is a substantial improvement of the `near-model completeness' results established in known cases (such as \cite{bz, vdd-dense, dg}) in that it achieves a decomposition of definable sets into \emph{unions} (instead of boolean combinations) of cones.
It also includes definable maps $f:M^n\to M$ for any $n$ (instead of only $n=1$). To illustrate the last point, let us consider the following example of a map for $n=1$ from \cite{vdd-dense}. Consider a dense pair $\la \cal M, P\ra$ of real closed fields and let $\alpha\not\in P$. So $\cal M$ could be the real field, $P$ the field of real algebraic numbers, and $\alpha=\pi$. Let $f:M\to M$ be the definable map  given by
$$f(x)=\begin{cases}
r & \text{if $x=r+\alpha s$ for some (unique) $r,s\in P$}\\
0 & \text{otherwise.}
\end{cases}
$$
It is easy to see that the graph of $f$ is dense in $M^2$, and hence $f$ is not as tame as an $\cal L$-definable map. However, \cite[Theorem 3]{vdd-dense} establishes that every definable map $f:M\to M$ is given by an $\cal L$-definable map off a small set (here, the $\cal L$-definable map is 0 and the small set is $P+ \alpha P$). A far reaching application of our structure theorem is the following  generalization of this phenomenon.\\

 \noindent\textbf{ Theorem \ref{stII}.} {\em  Every $A$-definable map $f: M^n\to M$ is given by an $\Cal L_{A\cup P}$-definable map off a set of large dimension $ < n$.}\\

We expect that this theorem will be useful  in the future and already manifest one of its immediate corollaries here. Namely, we answer a question by Dolich-Miller-Steinhorn \cite{dms2}: in  dense pairs, the graph of a \emph{$\emptyset$-definable} unary function is nowhere dense (Proposition \ref{prop-ndg}). A further application is obtained in \cite{el-sl}, where Theorem \ref{stII} is used to prove that in the case of $\WM=\la \cal M, P\ra$ with $P$ independent,  every definable group is definably isomorphic to a group definable in \cal M.



In Section \ref{sec-equaldim}, we  compare the large dimension of a definable set to the $\scl$-dimension coming from \cite{beg}. In \cite{beg}, the authors work under three similar tameness conditions on $\widetilde {\cal M}$ and prove that the \emph{small closure} operator $\scl$ defines a pregeometry under further assumptions on $\cal M$ (\cite[Corollary 77]{beg}). Here, we observe that those further assumptions are in fact unnecessary (Corollary \ref{scl-preg}) and derive the equivalence of the two dimensions (Proposition \ref{equaldim}), always.

In Section \ref{sec-groups}, we exploit this equivalence and set forth the analysis of groups definable in $\widetilde{\cal M}$. Indeed, making use of desirable properties of `$\scl$-generic' elements (Fact \ref{generic}), we achieve the following result.\\

\noindent\textbf{Local theorem for definable groups (\ref{arounda}).}
\emph{Let $G=\la G, *\ra$ be a definable group of large dimension $k$. Then for every $\scl$-generic element $a$ in $G$,  there is a $2k$-cone $C\sub G\times G$, whose topological closure contains  $(a, a)$, and on which the operation
\[
(x, y)\mapsto x* a^{-1}* y
\]
is given by an $\cal L$-definable map.}\\

\noindent We note  that an analogous local theorem for semi-bounded groups was proved in \cite[Theorem 6.3]{el-sbdgroups} and was then vitally used in the global analysis of semi-bounded groups in \cite{ep-sel2}. 
We  expect that the present local theorem will be as  crucial in forthcoming analysis of  definable groups in $\widetilde{\cal M}$, and we list a series of open questions in the end of Section \ref{sec-groups}. The ultimate goal would be to understand definable groups in terms of $\cal L$-definable groups and small groups (Conjecture \ref{conjecture}). Note that \cal L-definable groups have been exhaustively studied and are well-understood,  some of the main  results being proved in \cite{cp, eo, ep-sel2, elst, hpp, hp, pps}.\smallskip


We next indicate some of the key aspects of this paper. Both the definition of the large dimension, as well as that of a cone, are based on the notion of a \emph{supercone} given in Section \ref{sec-cones}, which in its turn is based on the notion of a \emph{large} subset of $M$ coming from \cite{beg} or \cite{dg}. Namely, a supercone $J$ in $M^n$ is defined, recursively on $n$, as a union of a specific family of large fibers over a supercone in $M^{n-1}$. The large dimension of a definable set $X$ is then the maximum $k$ such that a supercone from $M^k$ can be \emph{embedded} into $X$. The nature of this embedding is crucial: while the definition of the large dimension is given via a strong notion of embedding, proving its invariance under definable bijections in Corollary \ref{bijection} requires an equivalent definition via a weaker notion of embedding. We establish that equivalence in Corollary \ref{embed}.

Let us now describe the main idea behind the proof of the Structure Theorem in Section \ref{sec-str-thms} that also explains  the role of  large dimension in it and motivates all the  preparatory work done in Sections \ref{sec-small} and \ref{sec-cones}. The notion of a large/small set is defined in Section \ref{sec-prelim} and that of a $k$-cone in Section \ref{sec-cones}. Roughly speaking, a $k$-cone is a set of the form
$$h\left(\bigcup_{g\in S} \{g\}\times J_g\right),$$
where $h$ is an $\cal L$-definable continuous map with each $h(g, -)$  injective, $S\sub M^m$ is a small set, and $\{J_g\}_{g\in S}$ a definable family of supercones in $M^k$. The proof of the Structure Theorem runs by simultaneous induction on $n$ for three statements, Theorem \ref{str-thm-sets} (1) - (3). For (1), in the inductive step, let $X\sub M^{n+1}$. By the inductive hypothesis, we may assume that the projection $\pi(X)$ onto the first $n$ coordinates is a $k$-cone, and by definability of smallness (Remark \ref{definability}(a)), we may separate two cases. If all fibers of $X$ above $\pi(X)$ are large, then we can simply follow the definition of a cone and, using $(2)_n$ and  Lemma \ref{cone+1},  we conclude that $X$ is a $k+1$-cone.
If all fibers of $X$ above $\pi(X)$ are small, then we first need to turn $X$ into a small union of ($\cal L$-definable images of) subsets $J_g\sub \pi(X)$ as above. This is achieved using Lemma \ref{philippcor} and it is illustrated in Example \ref{exa-philippcor}. Unfortunately, the sets $J_g$  obtained are not necessarily supercones, but we can remedy the situation by applying a uniform version of $(1)_{n-1}$, namely $(3)_{n-1}$. We derive $(3)_n$ from $(1)_n$ using a standard compactness argument. We derive $(2)_n$ from $(1)_n$ by first applying Corollary \ref{philipp2uniform} to obtain $\cal L$-definability of $f$ outside a subset of $\pi(X)$ of \emph{smaller large dimension}. We then conclude it by sub-induction on large dimension.




In Section \ref{sec-optimality}, we explore the optimality of our Structure Theorem. We prove that a stronger version where the notion of a cone is strengthened by requiring that $h$ is injective on $\bigcup_{g\in S}\{g\}\times J_g$ is not possible. This is essentially due to the lack of definable choice in our setting (see, for example, \cite[Section 5.5]{dms1}).   In Section \ref{sec-future}, however, we isolate a key `choice property' that implies a strengthened version of Lemma \ref{philippcor} (see Lemma \ref{philippcor2}),  which in turn guarantees a Strong Structure Theorem (\ref{strongstrthm}). This study suggests a new line of research where the behavior of  $\Cal L$-definable maps on small sets is pending to be explored. A list of open questions is included, whereas further optimality results are established in subsequent work \cite{el-opt}.

It is an important feature of this work that we keep  track of all  parameters. If $X$ is an $A$-definable set then, by Lemma \ref{defB} below, its topological closure is $\cal L_{A\cup P}$-definable. However, our Structure Theorem establishes that every $A$-definable set is a finite union of $A$-definable sets (the cones) whose closures are actually $\cal L_A$-definable. We warn the reader that we make a slight abuse of terminology in the interests of keeping the text succinct: an $A$-definable cone will be assumed to have its closure $\cal L_A$-definable; see Section \ref{subsec-cones} for more details.\\

\noindent\textbf{Acknowledgements.}
The authors wish to thank Chris Miller, Rahim Moosa  and Ya'acov Peterzil for taking the time to answer their questions.
 The first two authors also wish to thank the Center of Mathematics and Fundamental Applications (CMAF) at the University of Lisbon, where this work began, and the  Foundation of Science and Technology (FCT) in Portugal for its kind support.

\section{The setting}\label{sec-prelim}


Throughout this paper, we fix an o-minimal theory $T$ expanding the theory of ordered abelian groups with a distinguished positive element $1$. We also fix the language $\Cal L$ of $T$ and $\Cal L(P)$ the language $\Cal L$ augmented by a unary predicate symbol $P$. Let $\widetilde{T}$ be an $\Cal L(P)$-theory expanding $T$. If $\cal M=\la M, <, +, \dots\ra\models T$, then $\widetilde{\cal M}=\la \cal M, P\ra$ denotes an expansion of \cal M that models $\widetilde{T}$. By `$A$-definable' we mean `definable in $\widetilde{\cal M}$ with parameters from $A$'. By `$\cal L_A$-definable' we mean `definable in \cal M with parameters from $A$'. We omit the index $A$ if we do not want to specify the parameters.


For a subset $X\subseteq M$, we write $\dcl(X)$ for the definable closure of $X$ in $\Cal M$, and $\dcl_{\cal L(P)}(X)$ for the definable closure of $X$ in $\widetilde{\Cal M}$. By the o-minimality of $T$, the operation that maps $X\subseteq M$ to $\dcl(X)$ is a pregeometry on $M$. For an $\cal L$-definable set $X\sub  M^n$, we denote by $\dim(X)$  the corresponding pregeometric dimension.


The following definition is taken essentially from \cite{dg}.

\begin{defn}\label{def-small}
 Let $X\sub M^n$ be a definable set. We call $X$ \emph{large} if there is some $m$ and an $\cal L$-definable function $f:M^{nm}\to M$ such that $f(X^m)$ contains an open interval in $M$. We call $X$ \emph{small} if it is not large.
\end{defn}


Note that if $X\subseteq M$ is small and $I$ an interval in $M$, then $I\sm X$ is large (with a proof identical to that of \cite[Lemma 20]{beg}). We will use this observation throughout this paper. In Lemma \ref{lem:beg2}  and Corollary \ref{internal} below we prove that smallness is equivalent to $P$-internality, in the usual sense of geometric stability theory.

\begin{defn}
If $X, Z\sub M^n$ are definable, we say that $X$ is \emph{small in $Z$} if $X\cap Z$ is small. We say that $X$ is \emph{co-small in $Z$} if $Z\sm X$ is small.
\end{defn}

\subsection{Assumptions}\label{sec-assumptions} We assume that $\widetilde T$ satisfies  the following three tameness conditions: for every model $\widetilde{\cal M}\models \widetilde{T}$,
\begin{itemize}

\item [(I)]  $P$ is small.

\item [(II)] (Near model-completeness) Every $A$-definable set $X\sub M^n$ is a boolean combination of sets of the form
\[
\{x\in M^n: \exists z\in P^m \varphi(x, z)\},
\]
 where $\varphi(x, z)$ is an $\cal L_A$-formula.

\item [(III)] (Open definable sets are $\cal L$-definable) For every parameter set $A$ such that $A\setminus P$ is $\dcl$-independent over $P$, and for every $A$-definable set $V \subset M^s$, its topological closure $cl(V)\subseteq M^{s}$ is $\cal L_A$-definable.

\end{itemize}

\noindent\textbf{From now on, and unless stated otherwise, $\widetilde T$ satisfies Assumptions (I)-(III) and $\widetilde{\cal M}=\la \cal M, P\ra$ is a sufficiently saturated model of $\widetilde T$.}

\begin{remark} (i)
Assumptions (I)-(III) are analogous to Assumptions (1)-(3) from \cite[Theorem 3]{beg}. Here, however, we insist on having some control on the defining parameters. Moreover, an easy argument  shows that under our assumptions, (3) from \cite[Theorem 3]{beg} holds, but without the additional condition that the set $S$ mentioned there be $\emptyset$-definable.

(ii) Assumption (III) indeed guarantees that open definable sets are $\cal L$-definable,  see Lemma \ref{defB} below.

(iii) We do not know whether assumptions (I) and (III) imply (II).


\end{remark}

\subsection*{Notation-terminology} The topological closure of a set $X\sub M^n$ is denoted by $cl(X).$ If $X, Y \subseteq M$ and $b=(b_1,\dots, b_n)$, we sometimes write $X \cup b$ or $Xb$ for $X \cup \{b_1,\dots, b_n\}$, and $XY$ for $X\cup Y$. If $\varphi(x,y)$ is an $\Cal L(P)$-formula and $a\in M^n$, then we write $\varphi(M^m,a)$ for
\[
\{ b \in M^m \ : \ \widetilde{\Cal M} \models \varphi(b,a)\}.
\]
Similarly, given any subset $X\subseteq M^m \times M^n$ and $a\in M^n$, we write $X_a$ for
\[
\{ b \in M^m \ : \ (b,a) \in X\}.
\]
For convenience, we sometimes write $f(t, X)$ for $f(\{t\}\times X)$.  
If $m\le n$, then $\pi_m:M^n\to M^m$ denotes the projection onto the first $m$ coordinates. We write $\pi$ for $\pi_{n-1}$, unless stated otherwise.
By an open box in $M^k$, or a $k$-box, we mean a set $I_1\times\dots\times I_k\sub M^k$, where each $I_j\sub M$ is an open interval. By dimension of an $\cal L$-definable set we mean its usual o-minimal dimension, and the notions of $\dcl$-independence, $\dcl$-rank and $\dcl$-generics are the usual notions attached to the $\dcl$-pregeometry (see, for example, \cite{pi-groups}).
A family $\cal J=\{J_g\}_{g\in S}$ of sets is called definable if $\bigcup_{g\in S}\{g\}\times J_g$ is definable,  disjoint if every two elements of it are disjoint, and small if $S$ is small. We often identify $\cal J$ with $\bigcup_{g\in S}\{g\}\times J_g$.  If for each $t\in T$, $\cal J_t=\{J_{g, t}\}_{g\in S_t}$ is a family of sets, we call $\{\cal J_t\}_{t\in T}$ definable if $\bigcup_{t\in T, g\in S_t}\{(g, t)\}\times J_{g, t}$ is definable.

Our examples are often given for structures over the reals (such as Example \ref{exa-lou1} and the counterexample in Section \ref{sec-optimality}). But they can easily be adopted to the current, saturated setting, by moving to an elementary extension.

\subsection{Examples}\label{sec-examples} $ $

\subsection*{Dense pairs} The first example we wish to consider is dense pairs of o-minimal structures. A \emph{dense pair} $\la \Cal M,\Cal N \ra$ is a pair of models of $T$ such that $\Cal N \neq \Cal M$, but $\Cal N$ is dense in $\Cal M$. Let $\widetilde T= T^d$ be the theory of dense pairs in the language $\cal L(P)$.
By \cite{vdd-dense}, $T^d$ is complete and every model of $T^d$ satisfies (I) and (II) (\cite[Lemma 4.1]{vdd-dense} and \cite[Theorem 1]{vdd-dense}, respectively).

It is left to explain why (III) holds in dense pairs. Here we apply \cite[Corollary 3.1]{bh}. Let $A$ be a parameter set such that $A\setminus N$ is $\dcl$-independent over $N$. Set
\[
D := \{ a \in M \ : \ a \hbox{ is $\dcl$-independent over $N\cup A$}\}.
\]
It is easy to see that $D$ and $A$ satisfy Assumptions (1) and (2) of \cite[Corollary 3.1]{bh}. It is left to show that also the third assumption of that corollary holds. Towards that goal, recall the following notation from \cite{vdd-dense}. Given $\Cal M,\Cal N, \Cal O, \Cal Q\models T$ with $\Cal M \subseteq \Cal N \subseteq \Cal Q$ and $\Cal M \subseteq \Cal O \subseteq \Cal Q$, we say that $\Cal N$ and $\Cal O$ are \emph{free over $\Cal M$} (in $\Cal Q$) if every subset $Y\subseteq  N$ that is $\dcl$-independent over $\Cal M$ is also $\dcl$-independent over $\Cal O$.

\begin{prop}\label{prop:con3} Let $a\in D$. Then the $\cal L(P)$-type of $a$ over $A$ is implied by the $\cal L$-type over $A$ and the fact that $a \in D$.
\end{prop}
\begin{proof} Let $\la \Cal M, \Cal N\ra\models T^d$ be $\kappa$-saturated, where $\kappa > |T|$. Let $\Gamma$ be the set of all isomorphisms $i: \la\Cal M_1,\Cal N_1\ra \to \la\Cal M_2,\Cal N_2\ra$ between substructures of $\la\Cal M,\Cal N\ra$ such that $| M_1|< \kappa, | M_2|<\kappa$, $\Cal M_1$ and $\Cal N$ are free over $\Cal N_1$ and $\Cal M_2$ and $\Cal N$ are free over $\Cal N_2$. By \cite[Claim on p. 67]{vdd-dense}, $\Gamma$ has the back-and-forth property.
Let $a,b \in D$ such that $a$ and $b$ satisfy the same $\Cal L$-type over $A$. Then there is an $\Cal L$-isomorphism
\[
i : \dcl(a \cup A) \to \dcl(b \cup A).
\]
Since both $a$ and $b$ are $\dcl$-independent over $N\cup A$, the isomorphism expands to an isomorphism
\[
i : \big\la\dcl(a \cup A),\dcl(A)\cap \Cal N\big\ra \to \big\la\dcl(b \cup A),\dcl(A)\cap \Cal N\big\ra
\]
of substructures of $\la\Cal M,\Cal N\ra$. Since $a \cup (A \setminus N)$ is $\dcl$-independent over $N$, $\dcl(a \cup A)$ and $\Cal N$ are free over $\dcl(N)\cap \Cal N$. By the same argument $\dcl(b \cup A)$ and $\Cal N$ are free over $\dcl(A)\cap \Cal N$. Hence $i \in \Gamma$. Since $\Gamma$ is a back-and-forth system, $a$ and $b$ satisfy the same $\cal L(P)$-type over $A$.
\end{proof}

\subsection*{Groups with the Mann property} Let $\Gamma$ be a dense subgroup of $\mathbb{R}_{>0}$ that has the \emph{Mann property}, that is for every $a_1,\dots,a_n \in \mathbb{Q}^{\times}$, there are finitely many $(\gamma_1,\dots,\gamma_n) \in \Gamma^n$ such that $a_1\gamma_1+\dots+a_n\gamma_n = 1$ and $\sum_{i\in I} a_i\gamma_i \neq 0$ for every  nonempty subset $I$ of $\{1,\dots,n\}$.  Every multiplicative subgroup of finite rank in $\R_{>0}$ has the Mann property, see \cite{ess}.

We assume  that for every prime number $p$, the subgroup of $p$-th powers in $\Gamma$ has finite index in $\Gamma$. Let $\Cal L$ be the language of ordered rings augmented by a constant symbol for each $\gamma \in \Gamma$. Let $T$ be the theory of $\la\R,(\gamma)_{\gamma \in \Gamma}\ra$ in that language and let $\widetilde T= T(\Gamma)$ be the theory of $\la\R,(\gamma)_{\gamma \in \Gamma},\Gamma\ra$ in the language $\Cal L(P)$. By \cite[Theorem 7.5]{dg}, every model of $T(\Gamma)$ satisfies (II). A proof that every model satisfies (I) is in \cite[Proposition 2.9]{gh}.

 Again, we  show that (III) follows from \cite[Corollary 3.1]{bh}. Let $\la\Cal M, P\ra\models T(\Gamma)$. Let $A$ for every parameter set $A$ such that $A\setminus P$ is $\dcl$-independent over $P$. Set
\[
D := \{ a \in M \ : \ a \hbox{ is $\dcl$-independent over } P \cup A\}.
\]
One can check easily that assumptions (1) and (2) of \cite[Corollary 3.1]{bh} follow from the o-minimality of $T$. Finally it is easy to see that almost the same proof as for Proposition \ref{prop:con3}, just using the back-and-forth system in the proof of Theorem 7.1 in \cite{dg} instead of \cite[Claim on p. 67]{vdd-dense}, shows that assumption (3) of \cite[Corollary 3.1]{bh} is satisfied as well.\\

There are several other closely related examples. In \cite{hi} proper o-minimal expansions $\Cal R$ of the real field and finite rank subgroups $\Gamma$ of $\mathbb{R}_{>0}$ are constructed such that the structure $(\Cal R,\Gamma)$ satisfies Assumptions (I)-(III). Indeed, the fact that these structures satisfy Assumptions (I) and (II) is immediate from results in \cite{hi}. Assumption (III) follows by the same argument as above. In \cite{bz,gh} certain expansions of the real field by subgroups of either the unit circle or an elliptic curve are studied. One can easily show using the above argument that these structures satisfy Assumptions (I)-(III) after adjusting their statements for the fact that $P$ now lies in a $1$-dimensional semialgebraic set in $\R^2$. Since no significant new argument is involved, we leave it to the reader to verify that our main results also hold in this slightly more general setting.

\subsection*{Independent sets}  Let $\widetilde T= T^{\textrm{indep}}$ be an $\Cal L(P)$-theory extending $T$ by axioms stating that $P$ is dense and $\dcl$-independent. By \cite{dms2}, $T^{\textrm{indep}}$ is complete and every model of $T^{\textrm{indep}}$ satisfies (I) and (II) by \cite[2.1]{dms2} and \cite[2.9]{dms2}, respectively. As usual, we  show that (III) follows from \cite[Corollary 3.1]{bh}. Let $\la\Cal M, P\ra\models T^{\textrm{indep}}$. Let $A$ be a parameter set such that $A\setminus P$ is $\dcl$-independent over $P$. Set
\[
D := \{ a \in M \ : \ a \hbox{ is $\dcl$-independent over } P \cup A\}.
\]
From the o-minimality of $T$, assumptions (1) and (2) of \cite[Corollary 3.1]{bh} follow easily as above. By \cite[2.12]{dms2}, assumption (3) of \cite[Corollary 3.1]{bh} holds as well.

\subsection*{Non-examples}$ $

(1) By Assumption (III), $P$ must be dense in a finite union of open intervals and points. Indeed, the closure of $P$ has to be $\cal L$-definable. Therefore, tame expansions of $\cal M$ by discrete sets, such as $\la \R, 2^{\bb Z}\ra$, do not belong to this setting.

(2) We do not know whether the theory of every expansion $\la \cal M, P\ra$ of an o-minimal structure $\cal M$ with o-minimal open core \cite{dms1, ms} satisfies Assumptions (II) or (III). Assumption (I) does not hold in case $P$ is a generic predicate.

(3) If $\widetilde{\cal M}=\la M, <, +, \cal P\ra$ is semi-bounded, that is, a pure ordered group expanded by the structure of a real closed field $\cal P=\la P, \oplus, \otimes\ra$  on some bounded open interval $P\sub M$, then Assumptions (II) and (III) hold  by \cite{el-sbdgroups}, but (I) does not.\\

\subsection{$\Cal L$-definability}

In general, an $\cal L$-definable set $X$ which is also $A$-definable need not be $\cal L_A$-definable. For example, let $\widetilde{\cal M}=\la \cal M, P\ra$ be a dense pair of real closed fields, and $x,y\in M\sm P$ such that there are (unique) $g,h\in P$ with $x = g + hy$. Then $\{g\}$ is $\cal L$-definable and $\{x, y\}$-definable, but in general not $\cal L_{\{x, y\}}$-definable. The following lemma, however, implies, in particular, that every such $X$ is always $\cal L_{A\cup P}$-definable.



\begin{lem}\label{defB} Let $X\sub M^n$ be an $A$-definable set. Then there is a finite $B\sub A$ such that $X$ is $B\cup P$-definable and $B$ is $\dcl$-independent over $P$. Hence, by Assumption (III), $cl(X)$ is $\cal L_{A\cup P}$-definable. In particular, if  $X$ is closed (or open), then it is $\cal L_{A\cup P}$-definable.
\end{lem}
\begin{proof} Without loss of generality, we can assume that $A$ is finite. Let $B\subseteq A$ be a maximal subset of $A$ that is $\dcl$-independent over $P$. Suppose $B=\{b_1, \dots, b_k\}$. Hence for every $a\in A\setminus B$, there are $g_a \in P^l$ and  an $\cal L_{\emptyset}$-definable map $h: M^{l+k} \to M$ such that
$h(g_a, b_1,\dots,b_k) = a$. Set $H = \{ g_a \ : \ a \in A\setminus B\}$.  Since $X$ is $A$-definable, it is also $B\cup H$-definable.
\end{proof}

A positive answer to the following open question would give better control to the set of parameters (see also  after Corollary \ref{cor:corp1} below).

\begin{question}\label{question-parameters}
  For $X$ as above, are there finite $B\sub A$ and $H\subseteq P\cap \dcl_{\cal L(P)}(A)$, such that $X$ is $B\cup H$-definable and $B$ is $\dcl$-independent over $P$?
\end{question}
By \cite[2.26]{dms2}, Question \ref{question-parameters} admits a positive answer when $\widetilde{T}=T^{\textrm{indep}}$. However, we do not know the answer even when  $\widetilde{T}=T^d$.

The reader might wonder whether for every definable subset $X$ of $P^l$  there is an $\Cal L$-definable set $Y\subseteq M^n$ such that $X=Y\cap P^n$. While this is true for dense pairs by \cite[Theorem 2(2)]{vdd-dense}, this fails in examples arising from groups with the Mann property (see \cite[Proposition 57]{beg}).

Although all our known examples that satisfy Assumptions (I)-(III) have NIP (see \cite{bdo,gh2}), the following question stands open.

\begin{question}
Do Assumptions (I)-(III) imply that $\widetilde T$ has NIP?
\end{question}

\subsection{Basic facts for $\cal L$-definable and small sets}


We include some basic facts that will be used in the sequel.

\begin{fact}\label{fin1}
Let $f:X\sub M^m\to M^n$ be a finite-to-one $\cal L$-definable function. Then there is a finite partition $X=X_1\cup\dots\cup X_k$ into definable sets such that each $f_{\res X_i}$ is injective.
\end{fact}
\begin{proof}
  Standard.
\end{proof}

\begin{fact}\label{fin2}
Let $f:A\sub M^m\to M^n$ be an $\cal L$-definable function. Let
$$X_f=\{a\in A: f^{-1}(f(a)) \text{ is finite}\}.$$
Then $\dim f(A\sm X_f)<\dim A$.
\end{fact}
\begin{proof}
Let $R=f(A\sm X_f)$. By definition of $X_f$, for every $r\in R$, $f^{-1}(r)$ has dimension $>0$. Since $A\sm X_f$ equals the disjoint union $\bigcup_{r\in R} f^{-1}(r)$, we have by standard properties of dimension:
$$\dim(A\sm X_f)\ge \min_r \dim f^{-1}(r) +\dim R.$$
Hence,  $\dim A\ge 1+\dim R$ and $\dim R<\dim A$.
\end{proof}



\begin{fact}\label{fact2}

 If  $X, Z, I\sub M^m$ are definable sets, and $X$ is co-small in $Z$, then $X\cap I$ is co-small in $Z\cap I$.



\end{fact}
\begin{proof}
Immediate from the definitions.
\end{proof}

\section{Small sets}\label{sec-small}

In this section we establish properties of small sets that will be important in the proof of the Structure Theorem. The two most crucial results are Lemma \ref{philippcor} and Corollary \ref{philipp2uniform} below.

\subsection{Families of small sets and $P$-boundness}
With the exception of Lemma \ref{philippcor} below, the results of this section were either established in \cite{beg} or are minor improvements of  results in \cite{beg}. Since the assumptions in \cite{beg} differ from ours, we reprove the results here. Most of the proofs are direct adjustments from those in \cite{beg}, but are included  for the convenience of the reader. They often involve induction on formulas whose base step deals with a `basic' set defined next.\smallskip


\begin{defn} A subset $X \subseteq M^n$ is called \emph{basic over $A$} if it is of the form
\[
\bigcup_{g \in P^m} \varphi(M,g),
\]
for some $\Cal L_A$-formula. We say $X$ is \emph{basic} if it is basic over some parameter set $A$.
\end{defn}

 Note that by Assumption (II) every definable set is a boolean combination of basic sets.

\begin{lemma}\label{lem:uniform small} Let $p\in \N$. For $j=1,\dots,p$, let $\{S_{1,j,t}\}_{t\in M^l},\{S_{2,j,t}\}_{t\in M^l}$ be $A$-definable families of subsets of $P^{n}$. Let $f_1,\dots, f_p,h_1,\dots, h_p: M^{n+l}\to M$ be $A$-definable functions. Then there are $A$-definable families $\{Q_{j,t}\}_{t\in M^l},\{R_{j,t}\}_{t\in M^l}$ of subsets of $P^n$, for $j=1,\dots, p$, such that for every $t\in M^l$,
\begin{align*}
\bigcup_{j} f_j(S_{1,j,t},t) \cap \bigcup_{j} h_{j} (S_{2,j,t}, t)&= \bigcup_{j} f_j(Q_{j,t},t),\\
\Big(M\setminus \bigcup_{j} f_{j}(S_{1,j,t},t)\Big) \cup \bigcup_{j} h_{j} (S_{2,j,t}, t)&= M \setminus \bigcup_{j} f_j(R_{j,t},t).
\end{align*}
\end{lemma}
\begin{proof} Set
\[
Q_{j,t} := \{ g \in S_{1,j,t} : \bigvee_{i=1}^p \exists g' \in  S_{2,i,t} \ h_i(g',t) = f_j(g,t) \}
\]
and
\[
R_{j,t} := \{ g \in S_{1,j,t} : \bigwedge_{i=1}^p \forall g' \in  S_{2,i,t} \ h_i(g',t) \neq f_j(g,t) \}.
\]
\end{proof}

\begin{lem}\label{lem:beg1} Let $\{X_t\}_{t\in M^l}$ be an $A$-definable family of subsets of $M$. Then there are $m,n,p \in \N$ and for each $i=1,\dots,m$ there are
\begin{itemize}
\item an $A$-definable family $\{S_{i,j,t}\}_{t\in M^l}$ of subsets of $P^n$, for each $j=1,\dots,p$,
\item $\Cal L_A$-definable functions $h_{i,1},\dots,h_{i,p} : M^{n+l} \to M$,
\item an $A$-definable function $a_i : M^l \to M\cup \{\infty\}$,
\end{itemize}
such that for $t \in M^l$,
\begin{itemize}
\item [(i)] $-\infty =a_0(t) \leq a_1(t) \leq \dots \leq a_m(t) =\infty$ is a decomposition of $M$, and
\item [(ii)] one of the following holds:
\begin{itemize}
  \item[(a)] $[a_{i-1}(t),a_i(t)] \cap X_t = V_{i,t}$ or
  \item[(b)] $[a_{i-1}(t),a_i(t)] \cap X_t = (M\setminus V_{i,t}) \cap [a_{i-1}(t),a_i(t)],$
\end{itemize}
where $V_{i,t} = \bigcup_{j} h_{i,j}(S_{i,j,t},t)$.
\end{itemize}
\end{lem}
\begin{proof} First consider a definable family of basic sets, say $(\mathbb{D}_t)_{t\in M^l}$, that is a definable family of the form
\[
\mathbb{D}_t = \bigcup_{g \in P^n} \varphi(M,g,t),
\]
where $\varphi(x,y,z)$ is an $\cal L_A$-formula and $t \in M^l$. By cell decomposition, there are two finite sets $J_1,J_2$, $\Cal L_A$-definable cells $(Y_{1,j})_{j\in J_1}$ and $(Y_{2,j})_{j\in J_2}$ in $M^{n+l}$ and $\Cal L_A$-definable functions
$(f_{1,j})_{j\in J_1}$, $(f_{2,j})_{j\in J_2}$ and $(f_{3,j})_{j\in J_2}$ from $M^{n+l}$ to $M$ such that
\[
\mathbb{D}_t = \bigcup_{j\in J_1} f_{1,j}(Y_{1,j,t} \cap P^n,t) \cup \bigcup_{j \in J_2} \bigcup_{g \in Y_{2,j,t} \cap P^n} \big(f_{2,j}(g,t),f_{3,j}(g,t)\big).
\]
Without loss of generality, we can assume that $|J_1|=|J_2|$. Set $p:= |J_1|$ and assume that $J_1=J_2=\{1,\dots,p\}$. Set $U_t:=\bigcup_{j \in J_2} \bigcup_{g \in Y_{2,j,t} \cap P^n} \big(f_{2,j}(g,t),f_{3,j}(g,t)\big)$. Note that $U_t$ is open. By Assumption (III), $U_t$ is a finite union of open intervals. Since finitely many intervals only have  finitely many endpoints and $U_t$ is $At$-definable, the endpoints of the intervals of $U_t$ are $At$-definable. Let $V_t$ be the topological closure of $\bigcup_{j\in J_1} f_{1,j}(Y_{1,j,t} \cap P^n,t)$. By Assumption (III) again, $V_t$ is $\Cal L$-definable. Hence it is a finite union of intervals and points. Since there are only finitely many endpoints and $V_t$ is $At$-definable, these endpoints are $At$-definable. Hence we have a decomposition of $M$
\[
-\infty =a_0(t) \leq a_1(t) \leq \dots \leq a_m(t) =\infty
\]
such that either
\begin{itemize}
\item $(a_{i-1}(t),a_i(t))\cap \mathbb{D}_t = (a_{i-1}(t),a_i(t))$ or
\item $(a_{i-1}(t),a_i(t))\cap \mathbb{D}_t = (a_{i-1}(t),a_i(t)) \cap \bigcup_{j\in J_1} f_{1,j}(Y_{1,j,t} \cap P^n,t)$.
\end{itemize}
In the first case set $S_{i,j,t}:=\emptyset$ and set $h_{i,j}(x,y)=0$ for all $(x,y) \in M^{n+l}$. In the second case set
\[
S_{i,j,t} := \{ g \in Y_{1,j,t} \cap P^n \ : \ f_{1,j}(g,t) \in   (a_{i-1}(t),a_i(t)) \},
\]
and set $h_{i,j}=f_{1,j}$. By compactness, we can find an $m\in \N$ that works for every $t \in M^l$. Hence (i)-(ii) holds for $(\mathbb{D}_t)_{t\in M^n}$.

By Assumption (II) it is enough to check that if the statement of the Lemma holds for two definable $(X_t)_{t\in M^l}$ and $(Z_t)_{t\in M^l}$, then it also holds for $(M\setminus X_t)_{t\in M^l}$ and $(X_t \cup Z_t)_{t\in M^l}$. So suppose that the statement holds for $(X_t)_{t\in M^l}$ and $(Z_t)_{t\in M^l}$. It is immediate that the conclusion holds for $(M\setminus X_t)_{t\in M^l}$ as well. It is easy to check that Lemma \ref{lem:uniform small} implies that the conclusion also holds for $(X_t \cup Z_t)_{t\in M^l}$. 
\end{proof}

\begin{remark}\label{definability}
The sets $V_{i,t}$ above are small, since $P$ is small (Assumption (I)). Hence:
\begin{enumerate}
\item[(a)]  the set
$$\{ t \in M^n \ : \ X_t \cap [a_{i-1}(t),a_i(t)] \hbox{ is small} \}$$ is equal to   $$\{ t \in M^n \ : \ X_t \cap [a_{i-1}(t),a_i(t)] = V_{i,t} \}.$$
Hence, it is $A$-definable. In particular, the set of all $t\in M^n$ such that $X_t$ is small is $A$-definable.

\item[(b)] the set of $(t, a_i(t))$ for which $X_t$ is small in $(a_{i-1}(t), a_{i}(t))$ is $A$-definable.

\end{enumerate}
\end{remark}

We will make use of the following consequence of Lemma \ref{lem:beg1}.

\begin{cor}\label{uniform small}
Let $\{X_t\}_{t\in I}$ be an $A$-definable family of subsets of $M$, where each $X_t\sub M$ is small and $I\sub M^n$. Then there are $m\in \N$, $\cal L_A$-definable continuous functions $h_j: V_j\sub M^{m+n}\to M$ and $A$-definable families $\{S_{j,t}\}_{t\in I}$ of sets $S_{j,t}\sub P^m$,  $j=1, \dots, p$, such that for every $t\in I$,  $X_t= \bigcup_{j} h_j (S_{j,t}, t)$.
\end{cor}
\begin{proof} Without requiring the continuity of the $h_j$'s, the statement is immediate from Lemma \ref{lem:beg1}. Now, to get the continuity, apply the cell decomposition theorem for o-minimal structures to get, for each $j$, cells $V_{j,1}, \dots, V_{j, s(j)}$ such that $h_j$ is continuous on each $V_{j, i}$. Let $S'_{j, i, t}:=S_{j, t}\cap V_{j, i}\sub P^m$. We have
$$X_t=\bigcup_{j,i} h_j(S'_{j,i,t}, t),$$
as required.
\end{proof}

The following example shows that in the last corollary the set $S_{j,t}$ has to depend on $t$.


\begin{example}\label{ex1} Let $\widetilde{\cal M}\models T^d$. For every $a\in M^{>0}$, let $X_a=P\cap (0, a)$, and $$X=\bigcup_{a\in M^{>0}} \{a\} \times X_a.$$
Let $h_j$ and $S_{j,a}$ be as in Corollary \ref{uniform small}, and assume towards a contradiction that all $S_{j, a}$'s equal some $S_j$.
So for every $a\in M^{>0}$,
\begin{equation}
\tag{$\ast$} (0, a)\cap P= \bigcup_j h_j(S_j, a).
\end{equation}
Take $p \in S_j$. By o-minimality, $h_j (p, -)$ is eventually continuous close to $0$. Since $h_j(p,M^{>0})\subseteq P$ by ($\ast$) and $P$ is codense in $M$,
 $h_j(p,-)$ is eventually constant close to $0$. That is, there is $a_p>0$ and $c_p\in P$, such that for every $0<a<a_p$, $h_j(p, a)=c_p$. Thus, if $0<a< c_p$, we have $h_j(p, a)=c_p\not\in (0, a)\cap P$, a contradiction.
\end{example}

We now derive a few corollaries of the above results. The next lemma shows how to turn a family $X=\{X_a\}_{a\in C}$ of small sets into a small family of subsets $Z_{g}$ of $C$. This will be a crucial step in the proof of the Structure Theorem. There,  we will further need to replace $Z_{ig}$ by ``cones'', which are defined in Section \ref{sec-cones}.

\begin{lemma}\label{philippcor}
Let $X=\bigcup_{a\in C} \{a\}\times X_a$ be $A$-definable where each $X_a\sub M$ is small, non-empty, and $C\sub M^n$. Then there are $l, m\in \N$, and for each $i=1,\dots,l$,
\begin{itemize}
 \item an $\cal L_A$-definable continuous function $h_i: V_i \subseteq M^{m+n}\to M^{n+1}$,
 \item an $A$-definable small set $S_i\sub M^m$, and
 \item an $A$-definable set $Z_i\subseteq S_i \times C$ contained in $V_i$,
\end{itemize}
  such that for
\[
 U_i=h_i\left(\bigcup_{g\in S_i} \{g\}\times Z_{ig}\right)
 \]
we have
\begin{enumerate}
\item $X=U_1\cup\dots\cup U_l$ is a disjoint union,
\item for every $i$ and $g\in S$, $h_i(g, -):V_{ig}\sub M^n\to M^{n+1}$ is injective,
\item $C=\bigcup_{i, g} Z_{ig}.$
\end{enumerate}
\end{lemma}

\begin{proof}
 We first observe that there are $m, p\in \N$, $\cal L_A$-definable continuous functions $h_i:V_i\sub M^{m+n}\to M$ and $A$-definable families $Y_i$ of small sets $Y_{ia}\sub P^m$,  $i=1, \dots, p$, such that for every $a\in I$,
\begin{enumerate}
\item $X_a= \bigcup_{i} h_i (Y_{ia}, a)$
\item $\{h_i(Y_{ia}, a)\}_{i=1, \dots, p}$ are disjoint.
\end{enumerate}
Indeed, this follows from Corollary \ref{uniform small}; for (2), recursively replace $Y_{ia}$, $1<i\le p$, with the set consisting of all $z\in Y_{ia}$ such that $h_i(z, a)\not\in h_j(Y_{ja}, a)$, $0<j<i$.
We now have:
$$X=\bigcup_{a\in C} \{a\} \times X_a =\bigcup_i \bigcup_{a\in C}  \{a\} \times h_i (Y_{ia}, a).$$
For every $i$, let $S_i=P^m$. For every $i$ and $g\in P^m$, let
$$U_i=\bigcup_{a\in C}  \{a\} \times h_i (Y_{ia}, a),$$
which are also disjoint, and
$$Z_{ig}=\{a\in C: g\in Y_{ia}\}.$$
Since $h_i$ and $\{Y_{ia}\}_{a\in C}$ are $A$-definable, so are  $U_i$ and $\{Z_{ig}\}_{g\in S_i}$. We have $C=\bigcup_{i, g} Z_{ig}$. Consider now the $\cal L_A$-definable continuous map $\hat h_i: V_i\sub M^{m+n}\to M^{n+1}$ with
$$\hat h_i(g, a)=\left(a, h_i(g, a)\right).$$
Then
$$ U_i=\hat h_i\left(\bigcup_{g\in S_i} \{g\}\times Z_{ig}\right)$$
 works.
\end{proof}

\begin{remark}\label{SG}
As the last proof shows, in fact we obtain $S_i=P^m$. We decided, however, to keep the current formulation because the proof can then be adopted in similar situations (such as in Lemma \ref{philippcor2} below). Had we kept the stronger formulation ($S_i=P^m$), what follows would result to a  Structure Theorem \ref{str-thm-sets} where in Definition \ref{def-cone} of a cone we could require $S\sub P^m$. However, we recover this information anyway, see Remark \ref{rem-supercone}(7).
\end{remark}

Let us illustrate Lemma \ref{philippcor} with an example.

\begin{example}\label{exa-philippcor} Let $\widetilde{\cal M}\models T^d$.
For every $a\in M^{>0}$, let $X_a=P\cap (0, a)$, and
$$X=\bigcup_{a\in M^{>0}} \{a\} \times X_a.$$
Then we can turn $X$ into a small union of ($\cal L$-definable images of) large subsets of $M$, as follows. For every $g\in P$, let
$$J_g=\{a\in M: a>g\}.$$
Then
$$X=h\left(\bigcup_{g\in P} \{g\}\times J_g\right),$$
where $h:M^2\to M^2$ switches the coordinates, $h(x, y)=(y, x)$. In this case, $X$ is in fact seen to be $1$-cone (according to Definition \ref{def-cone} below).
\end{example}

We now turn to examine better the notion of smallness.

\begin{defn} A set $X \subseteq M^n$ is \emph{$P$-bound} over $A$, if there is an $\Cal L_A$-definable function $f: M^m \to M^n$ such that $X \subseteq f(P^m)$. We omit $A$ if we do not want to specify the parameters.
\end{defn}

\begin{lem}\label{lem:beg2} An $A$-definable set is small if and only if it is $P$-bound over $A$.
\end{lem}
\begin{proof} Since $P$ is small, it follows immediately that every $P$-bound set is small.  For the other direction, observe first that, by Corollary \ref{uniform small}, every $A$-definable small subset of $M$ is $P$-bound over $A$. Now let $X\sub M^n$ be $A$-definable, and let $\pi_i : M^n \to M$ be the projection onto the $i$-th coordinate. If $X$ is small, so is $\pi_i(X)$ for $i=1,\dots, n$. Since each $A$-definable small subset of $M$ is $P$-bound over $A$, so is $\pi_i(X)$. Hence $\prod_{i=1}^{n} \pi_i(X)$ is $P$-bound over $A$ and so is $X \subseteq \prod_{i=1}^{n} \pi_i(X)$.
\end{proof}

We  show that in the definition of largeness and $P$-boundedness, we can replace $\cal L$-definability by definability. Recall from geometric stability theory that given two definable sets $X\sub M^n$ and $Y\sub M^k$, $X$ is called \emph{$Y$-internal} over $A$ if there is an $A$-definable $f:M^{mk}\to M^n$ such that $X\sub f(Y^m)$.

 \begin{cor}\label{internal}  Let $X$ be a definable set.
\begin{enumerate}
\item  $X$ is $P$-bound over $A$ if and only if it is $P$-internal over $A$.
\item $X$ is large if and only if an open interval is $X$-internal.
\end{enumerate}
 \end{cor}
 \begin{proof}
 By Lemma \ref{lem:beg2}, Definition \ref{def-small} and Assumption (I), it is easy to see that (1) implies (2). For (1), let $F: M^k \to M^n$ be $A$-definable such that  $X \subseteq F(P^k)$.  Without loss of generality, we may assume that $ A\setminus P$ is $\dcl$-independent over $P$. For each $g\in P^k$, the singleton $\{ F(g) \}$ equals its topological closure. Since $F(g)$ is definable over $A \cup g$ and $(A\cup g)\setminus P$ is $\dcl$-independent over $P$, we get by Assumption (III) that $\{F(g)\}$ is $\cal L_{A\cup g}$-definable. Hence, by compactness, there are finitely many $\cal L_A$-functions $F_1, \dots, F_l$ such that for all $g\in P^k$, $F(g) = F_i(g)$ for some $i$. Hence
$$F(P^k) \subseteq \bigcup_i F_i(P^k).$$
However, the right hand side is $P$-bound over $A$, and hence so is $F(P^k)$.
 \end{proof}

The following is then immediate.
\begin{cor}\label{fact1}
Let $f:X\to M^n$ be a definable injective function. Then $X$ is small if and only if $f(X)$ is small.
\end{cor}

A stronger version of the Corollary \ref{fact1} is provided by the invariance result in Corollary \ref{bijection} below. Here are three more corollaries of Lemma \ref{lem:beg2}.

\begin{cor}\label{smallunion} Let $Y\subseteq M^m$ be small and let $(X_t)_{t\in Y}$ be a definable family of small sets of $M^n$. Then
$\bigcup_{t \in Y} X_t$
is small.
\end{cor}
\begin{proof} By Lemma \ref{lem:beg2} and compactness, there is a definable family of $\Cal L$-definable functions $(f_t)_{t \in Y}$ such that $X_t \sub f_t(P^k)$ for each $t\in Y$. Again by Lemma \ref{lem:beg2}, there is also an $\Cal L$-definable function $g : P^l \to M^m$ such that $Y \subseteq g(P^l)$. Set $h: M^{k+l} \to M^{n}$ be the function that takes $(x, y)$ to $f_{g(y)}(x)$. Then $\bigcup_{t \in Y} X_t \subseteq h(P^{k+l})$ and hence is $P$-bound.
\end{proof}

\begin{cor}\label{cor-union}
The union and cartesian product of  finitely many small sets is small.
\end{cor}
\begin{proof}
Immediate from Lemma \ref{lem:beg2}, Corollary \ref{smallunion} and the definitions.
\end{proof}


In the case of dense pairs, we obtain the following interesting result.

\begin{cor}\label{small0} Assume $\widetilde{\cal M}=\la \cal M, P\ra$ is a dense pair. Then every $\emptyset$-definable small set $X\sub M^n$ is contained in $P^n$.
\end{cor}
\begin{proof} By Corollary \ref{lem:beg2}, there is an $\cal L_\emptyset$-definable $f:M^m\to M^n$, such that $X\sub f(P^m)$. By \cite[Lemma 3.1]{vdd-dense}, $X\sub P^n$.
\end{proof}

\subsection{Definable functions outside small sets}\label{sec-functions-small}

 In this section we  analyze the behavior of definable functions outside small, or rather \emph{low}, sets. 
Note that Assumption (II) is not used in this section.


\begin{defn}
 We denote by $I_n(A)\subseteq M^n$ the set of all tuples $a=(a_1,\dots,a_n)\in M^n$ that are $\dcl$-independent over $P\cup A$.
\end{defn}

\begin{remark} (1)  Note that $I_n(A)$ is $\Cal L(P)_A$-type definable. Indeed, $a\in I_n(A)$ if and only if for all $0\le i<n$, $m, l\in \N$ and $\Cal L_A$-$(l+i)$-formula $\varphi(x, y)$, $a$ satisfies:
$$ \forall g\in P^l \, [ \text{if $\varphi(g, a_1, \dots, a_{i-1}, -)$  has $m$ realizations, then  $\models \neg\varphi(g, a_1, \dots, a_{i})$}].$$
(2) It is obvious that $I_n(A) = I_n(A\cup P)$ and $I_n(B) \supseteq I_n(A)$ for $B\subseteq A$.
\end{remark}

\begin{lem}\label{lem:plem1} Let $A\sub M$ that $A\setminus P$ is $\dcl$-independent over $P$, and let $\varphi(x,y,z)$ be an $\Cal L(P)_A$-formula. Then there are $\Cal L_A$-formulas $\psi_1(x,y,z),\dots,\psi_k(x,y,z)$ such that for all $a\in I_m(A)$ and $b\in P^n$ there is $i\in \{1,\dots,k\}$ with
\[
cl(\varphi(a,b,M^l)) = \psi_i(a,b,M^l).
\]
\end{lem}
\begin{proof} Let $a=(a_1,\dots,a_m)\in I_m(A)$ and $b=(b_1,\dots,b_n)\in P^n$. It follows that
\[
\left(A \cup \{a_1,\dots,a_m\}\cup \{b_1,\dots,b_n\}\right)\setminus P
\]
 is $\dcl$-independent over $P$. Since $I_m(A)$ is $\Cal L(P)_A$-type definable and $P$ is definable, the statement of the lemma follows from compactness and Assumption (III).
\end{proof}

\begin{prop}\label{lem:pprop1} Let $F:M^m \times M^n \to M$ be $A$-definable. Then there are $\Cal L_{A\cup P}$-definable continuous functions $F_i: Z_i \subseteq M^m \times M^n \to M$, $i=1,\dots,k$, such that
for all $a\in I_m(A)$ and $b\in P^n$ there is $i\in \{1,\dots,k\}$ with $(a,b)\in Z_i$ and $$F(a,b) = F_i(a,b).$$
Moreover, if $A$ is $\dcl$-independent over $P$, then the $F_i$'s can be chosen to be $\cal L_A$-definable.
\end{prop}
\begin{proof} By Lemma \ref{defB}, there is a finite $B\sub A$ such that $B$ is $\dcl$-independent over $P$ and $F$ is $B\cup P$-definable. So $(B\cup P)\sm P$ is also $\dcl$-independent over $P$. Let $\varphi(x,y,z)$ be an $\cal L(P)_{B\cup P}$-formula that defines the graph of $F$.  Hence by Lemma \ref{lem:plem1} there are $\Cal L_{B\cup P}$-formulas $\psi_1(x,y,z),\dots, \psi_k(x,y,z)$ such that for all $a\in I_m(B)$ and $b\in P^n$ there is $i\in \{1,\dots,k\}$ with
\[
cl({\varphi(a,b,M)}) = \psi_i(a,b,M).
\]
Since $\varphi(a,b,M)$ is a single point, we have $\varphi(a,b,M)=\psi_i(a,b,M)$. Define $F_i : M^{m+n} \to M$ such that $F_i(a,b)$ is the unique $c\in M$ with $\psi_i(a,b,c)$ if such $c$ exists, and $0$ otherwise. Since $\psi_i$ is an $\Cal L_{B\cup P}$-formula, $F_i$ is $\Cal L_{B\cup P}$-definable.
Thus we have $\Cal L_{A\cup P}$-definable functions $F_1,\dots,F_k: M^{m+n} \to M$, such that for all $a\in I_m(A)$ and $b\in P^n$ there is $i\in \{1,\dots,k\}$ such that $F(a,b) = F_i(a,b)$. Using cell decomposition in o-minimal structures, we can find an $\Cal L_{A\cup P}$-cell decomposition $C_1,\dots,C_l$ of $M^{m+n}$ such that each $F_i$ is continuous on each $C_j$. The conclusion of the lemma now holds with the $kl$-many functions of the form $F_i|_{C_j}$, where $i=1,\dots,k$ and $j=1,\dots,l$.

For the `moreover' clause, if $A\sm P$ is $\dcl$-independent over $P$,  we need not replace $A$ by $B\cup P$ in the above proof, which then shows that no further parameters from $P$ are needed.
\end{proof}

\begin{cor}\label{cor:pprop1} Let $F:P^n \to M$ be $A$-definable. Then there are $t\in \N$,  $\Cal L_{A}$-definable continuous functions $F_i: Z_i \subseteq M^{t+n} \to M$ with $Z_i$ a cell, $i=1,\dots,k$, and $u \in P^t$, such that for all $b\in P^n$ there is $i\in \{1,\dots,k\}$ with $(u,b)\in Z_i$ and $$F(b) = F_i(u,b).$$
\end{cor}
\begin{proof} By Proposition \ref{lem:pprop1} there are  $\Cal L_{A\cup P}$-definable continuous function $H_i: Y_i \subseteq M^{n} \to M$, $i=1,\dots, k$,  such that for every $b\in P^n$ there is $i\in \{1,\dots,k\}$ with $b\in Y_i$ and $F(b) = H_i(b)$. Now take $u \in P^t$ such that each $H_i$ is $\Cal L_{Au}$-definable.  For  $i=1,\dots,k$, pick an $\Cal L_A$-definable function $F_i : Z_i \subseteq M^{t+n}\to M$ such that $(Z_i)_u= Y_i$ and $F_i(u,b)=H_i(b)$ for each $b\in  Y_i$. By applying cell decomposition to $M^{t+n}$, we may further assume that each $F_i$ is continuous and $Z_i$ is a cell.
\end{proof}

A slightly weaker version of Corollary \ref{cor:pprop1} is known for dense pairs \cite[Theorem 3(3)]{vdd-dense}.

\begin{defn}\label{def-low} We call $X\subseteq M^n$, $n>0$, \emph{low over $B$} if there is $i\in\{1,\dots,n\}$ and $\Cal L_B$-definable function $f: M^{n-1} \times M^l \to M$ such that
\[
X = \{ (a_1,\dots,a_n) \in M^n \ : \ \exists g \in P^l \ f(a_{-i},g) = a_i\},
\]
where $a_{-i}=(a_1,\dots,a_{i-1},a_{i+1},\dots,a_n)$.
\end{defn}

Note that if a set $X\subseteq M$ is low, then it is small and co-dense in $M$. Generalizations of this statement are obtained in Lemmas \ref{low1} and \ref{nonlow} below.

\begin{cor}\label{cor:corp1} Let $F:M^n \to M$ be $A$-definable. Then there are $k, m, t\in\N$ and
\begin{itemize}
\item sets $X_j\subseteq M^n$ low over $A$, $j=1,\dots,k$,
\item  $\Cal L_{A}$-definable continuous functions $F_i: Z_i \subseteq M^{t+n}  \to M$, $i=1,\dots,m$,
\item $u \in P^t$, 
\end{itemize}
such that  for every $a\in M^n \setminus \bigcup_{j=1}^{k} X_j$, there is $i\in \{1,\dots,m\}$ with $a \in Z_i$ and
$$F(a) = F_i(u,a).$$
\end{cor}
\begin{proof} Note that $a \notin I_n(A)$ if and only if there are $i\in \{1,\dots, n\}$, an $\Cal L_A$-definable function $f: M^{l+(n-1)} \to M$ and $g \in P^l$,  such that $f(a_{-i},g) = a_i$. Hence $a\notin I_n(A)$ if and only if there is $X$ low over $A$ such that $a \in X$. By compactness and Proposition \ref{lem:pprop1}, there are $k,m\in \N$ and
\begin{itemize}
\item  $\Cal L_{A\cup P}$-definable functions $H_i: Y_i \subseteq M^{n} \to M$, $i=1,\dots,m$
\item sets $X_j \in M^n$ low over $A$, $j=1,\dots,k$,
\end{itemize}
such that for every $a\in M^n\sm (\bigcup_{j=1}^k X_j)$ there is $i\in \{1,\dots,m\}$ with $a\in Y_i$ and $F(a) = H_i(a)$. Now take $u \in P^t$ such that each $H_i$ is $\Cal L_{Au}$-definable, and continue as in the proof of Corollary \ref{cor:pprop1}.
\end{proof}

\begin{remark} It is natural to ask whether the extra parameter $u\in P^t$ in Corollary \ref{cor:corp1} can be chosen to be in $\dcl_{\cal L(P)}(A)$. When the answer to Question \ref{question-parameters} is positive, then the same proof gives that $u$ is $\cal L_{A\cup H}$-definable, for some  $H\sub P\cap \dcl_{\cal L(P)}(A)$. So in particular, this holds when $\tilde{T}=T^{\textrm{indep}}$ (independent set). When $\tilde{T}=T^d$ (dense pairs), we do not know the answer.
\end{remark}

\begin{remark}\label{Aind}
If $A\sm P$ is $\dcl$-independent over $P$, then using the `moreover' clause of Proposition \ref{lem:pprop1}, we can see that in Corollaries  \ref{cor:pprop1} and \ref{cor:corp1}, we obtain $t=0$ and $u$ be the empty tuple.
\end{remark}

Since low subsets of $M$ are small, we can easily get the following corollary of \ref{cor:corp1}. This corollary is  already known for $\widetilde T=T^d$ by \cite{vdd-dense}, with the aforementioned  control in parameters also established in \cite[Lemma 5]{tyne}. We omit its proof since it is in fact a special case of Theorem \ref{stII}(2) below. 


\begin{cor}\label{unary} Let $f:M\to M$ be $A$-definable. Then $f$ agrees off some small set with an $\cal L_{A\cup P}$-definable function $F:M\to M$.
\end{cor}

The Structure Theorem below is intended, among others,  to generalize this corollary to arbitrary definable maps $f:X\sub M^n\to M$ (see Theorem \ref{stII}(2)). For the moment, using compactness, we directly get the following uniform version of Corollary \ref{cor:corp1}.

\begin{cor}\label{philipp2uniform} Let $f:Z\times M^n \subseteq M^{m+n}\to M$ be an $A$-definable map.  Then there are $p,t \in \N$ and for each $i=1,\dots,p$ there are
\begin{itemize}
\item an $A$-definable family $\{X^i_{z}\}_{z\in Z}$ of low subsets of $M^n$,
\item an $\Cal L_A$-definable continuous function $f_i : Z_i \subseteq M^m \times P^t \times M^n \to M$,
\end{itemize}
such that for all $z\in Z$ there is $u \in P^t$ such that for all $a\in M^n \sm \bigcup_{i} X^i_{z}$, there is $i\in \{1,\dots, p\}$ with
\[
f(z, a) = f_i(z,u,a).
\]
\end{cor}
\begin{proof} The corollary follows easily from compactness and Corollary \ref{cor:corp1}.
\end{proof}

\section{Cones and large dimension}\label{sec-cones}

In this section, we introduce and analyze the two main objects of the paper, cones and large dimension.

\subsection{Cones}\label{subsec-cones} As mentioned in the introduction, the notion of a cone is based on that of a supercone, which in its turn generalizes the notion of being co-small in an interval. Both notions, supercones and cones, are unions of specific families of sets, which not only are definable, but they are so in a very uniform way. The definitions appear to be quite technical in the beginning, but as it turns out they are in fact optimal in several ways (see Section \ref{sec-optimality}, Question \ref{productcones} and \cite{el-opt}).

\begin{defn}[Supercones]\label{def-supercone}
A \emph{supercone} $J\sub M^k$, $k\ge 0$, and its \emph{shell} $sh(J)$ are defined recursively as follows:
\begin{itemize}
\item $M^{0}=\{0\}$ is a supercone, and $sh(M^{0})=M^{0}$.

\item A definable set $J\sub M^{n+1}$ is a supercone if $\pi(J)\sub M^n$ is a supercone and there are  $\cal L$-definable continuous maps $h_1, h_2: sh(\pi(J))\to M\cup \{\pm\infty\}$ with $h_1<h_2$, such that for every $a\in \pi(J)$, $J_a$ is contained in $(h_1(a), h_2(a))$ and it is co-small in it. We let $sh(J)=(h_1, h_2)_{sh(\pi(J))}$.
\end{itemize}
Abusing terminology, we say that a supercone $J$ is \emph{$A$-definable} if $J$ is an $A$-definable set and its closure is $\Cal L_A$-definable.
\end{defn}

Note that $sh(J)$ is the unique open cell in $M^k$ such that $cl(sh(J))=cl(J)$. That is, $sh(J)$ is the interior of $cl(J)$. Moreover, for every $a\in \pi(J)$, $J_a$ is contained in $U_a$ and it is co-small in it.


We remind the reader that in our notation we identify a family $\cal J=\{J_g\}_{g\in S}$ with $\bigcup_{g\in S}\{g\}\times J_g$. In particular, $cl(\cal J)$ and $\pi_n(\cal J)$ denote the closure and a projection of that set, respectively.

\begin{defn}[Uniform families of supercones]\label{def-uniform}
Let $\Cal J = \bigcup_{g\in S} \{g\}\times J_g\sub M^{m+k}$ be a definable family of supercones. We call $\Cal J$ \emph{uniform} if there is a cell $V\sub M^{m+k}$ containing $\cal J$, such that for every $g\in S$ and  $0<j\le k$,
$$cl(\pi_{m+j}(\cal J)_g)=cl(\pi_{m+j}(V)_g).$$
We call such a $V$ a \emph{shell} for $\cal J$. Abusing terminology, we call a uniform family \emph{$A$-definable}, if it is an $A$-definable  family of sets and  has an $\cal L_A$-definable shell.
\end{defn}

A shell for $\cal J$ need not be unique. It is, however, canonical in the sense of Lemma \ref{shell-canonical} below. Note also that if $\cal J$ is uniform, then so is each projection $\pi_{m+j}(\cal J)$. 








\begin{defn}[Cones]\label{def-cone}
A set $C\sub M^n$ is a \emph{$k$-cone}, $k\ge 0$,  if there are a definable small $S\sub M^m$, a uniform family $\cal J=\{J_g\}_{g\in S}$ of supercones in $M^k$, and an $\cal L$-definable continuous function $h:V \subseteq M^{m+k}\to M^n$, where $V$ is a shell for $\cal J$, such that
\begin{enumerate}
\item $C=h(\cal J)$, and
\item for every $g\in S$, $h(g, -): V_g\sub M^k\to M^n$ is injective.
\end{enumerate}
A \emph{cone} is a $k$-cone for some $k$. 
Abusing terminology, we call a cone $h(\cal J)$ \emph{$A$-definable} if $h$ is $\Cal L_A$-definable and  $\Cal J$ is  $A$-definable.
\end{defn}


\begin{defn}[Fiber $\cal L$-definable maps]\label{defn:fiber}
Let $C=h(\cal J)\sub M^n$ be a $k$-cone with $\cal J\sub M^{m+k}$, and $f: D\to M$ a definable function with $C\sub D$. We say that $f$ is \emph{fiber $\cal L$-definable with respect to $C$} if there is an $\Cal L$-definable continuous function $F: V\sub M^{m+k}\to M$, where $V$ is a shell for $\cal J$, such that
\begin{itemize}
\item $(f\circ h)(x)=F(x)$, for all $x\in \cal J$.
\end{itemize}
We call $f$ \emph{fiber $\cal L_A$-definable with respect to $C$} if $F$ is $\cal L_A$-definable. 
\end{defn}

\begin{remark}\label{rem-supercone}$ $
\begin{enumerate}
   \item If $J\sub M^n$ is a supercone, then $\pi_m(J)$ is a supercone, and for every $t\in \pi_m(J)$, $J_t$ is a supercone with closure $cl(J)_t$.


\item Let $\{X_t\}_{t\in Z}$ be an $A$-definable family of subsets of $M^n$, $\{U_t\}_{t\in Z}$ an $\cal L_A$-definable family of subsets of $M^n$, and $\{C_t\}_{t\in Z}$  an $A$-definable family of cones in $M^n$.  Using Remark \ref{definability}(a), it is not hard to see that the sets
    \begin{itemize}
\item $\{ t \in Z \ : \ X_t \hbox{ is a supercone with closure $cl(U_t)$}\}$
\item $\{ t \in Z \ : X_t \hbox{ is a cone}\}$
    \end{itemize}
    are both $A$-definable.

  \item The $0$-cones are exactly the small sets. Low subsets of $M^n$ (Definition \ref{def-low}) are $n-1$ cones, but not every $(n-1)$-cone is low.


  \item The terminology of $f$ being fiber $\cal L_A$-definable with respect to $C=h(\cal J)$ is justified by the fact that, in that case, for every $g\in \pi(\cal J)$, $f$ agrees on $h(g, J_g)$ with an $\cal L_{Ag}$-definable map; namely $F\circ h(g, -)^{-1}$. But we require further that the family of  these $\cal L_{Ag}$-definable maps is actually $\cal L_A$-definable and continuous. We illustrate this last point with Example \ref{exa-fiber} below. The same example also shows that the notion of being fiber $\cal L$-definable depends on $h$ and $\cal J$.

\item
It is easy to see that if $C=h(\Cal J)$ is an $A$-definable $k$-cone and $f: C\to M$ fiber $\Cal L_A$-definable with respect to $C$, then the graph of $f$ is an $A$-definable $k$-cone. We will not make use of this fact.

\item The closure of an $A$-definable cone $h(\cal J)$ is $\cal L_A$-definable. Indeed, if $h:V\sub M^{m+k}\to M^n$ is as in Definition \ref{def-cone}, then it is easy to check that $cl(h(\cal J))=cl(h(cl(\cal J)\cap V))$.

\item We may replace $S$ by a definable subset of $P^l$ in the definition of a cone $C$. Indeed, let $h:V\sub M^{m+k}\to M^n$ be as in that definition. Since
    $S$ is $P$-bound, there is an $\cal L$-definable $f:M^l\to M^m$ with $f(P^l)\supseteq S$. By partitioning $C$ into finitely many cones, we may assume that for some cell $Z\sub M^l$, $f:Z\to \pi(V)$ is  continuous and $S\sub f(Z\cap P^l)$. 
    So we may  replace $S$ by $S':=f^{-1}(S)\cap Z\sub P^l$, and $h$ by $H: \bigcup_{g\in Z} \{g\}\times V_{f(g)}\to M^n$ with $H(x, y)=h(f(x), y)$. We decided, however, to keep the current definition because we can then adopt it in similar situations (such as Theorem \ref{strongstrthm} below). See also Remark \ref{SG}.
\end{enumerate}
\end{remark}
\begin{example}\label{exa-fiber}
Consider a dense pair $\la \cal M, P\ra$ of real closed fields and let $S=P+aP$ for some $a\not\in P$. The following map is taken from \cite{vdd-dense}.  Let $f:S\to M$ be the $a$-definable map  given by $f(x)=r$, where $x=r+sa$ for some (unique) $r,s\in P$.  Then, clearly, for every $x=r+sa\in S$, $f(x,-):M^0\to M$ agrees with the $\cal L_{r}$-definable map $H_r$  map given by $H_r(x, -)=r$. However, the family of maps $H_r$ is not $\cal L$-definable. Now re-write $S$ as the $a$-definable cone
$$S\times M= h(P^2)$$
where $h(p,q)=p+aq$, and let $F:M^2\to M$ be the projection onto the first coordinate. Then,  for every $(p, q)\in P^2$, we have
$$f(p+aq)=p=F(p,q),$$
witnessing that  $f$ is fiber $\cal L_a$-definable with respect to $h(P^2)$.
\end{example}

 We next observe several easy consequences of the definitions that will be used in the proof of the Structure Theorem. The first lemma  draws a  connection between cones and the $\dcl$-$\rank$ over tuples over $P$. Further results of this sort will be explored in Section \ref{sec-equaldim}. 

\begin{lem}\label{lem-rankcone} Let $a \in M^n$ and $A\sub M$. Then
$$\dcl\text{-}\rank(a/AP)=\min \{k\in \N: \text{$a$ is contained in an $A$-definable $k$-cone}\}.$$
\end{lem}
\begin{proof} ($\le$). This follows easily from the definition of a $k$-cone.

\noindent ($\ge$). Let $a=(a_1,\dots,a_n)$ and set $k=\dcl$-$\rank(a/AP)$. We will find an $A$-definable $k$-cone that contains $a$. Without loss of generality, we may assume that \linebreak$\dcl$-$\rank((a_1,\dots, a_k)/AP) = k.$ Hence there are an $\Cal L_A$-definable $Z \subseteq M^{l+n}$ and $s \in P^l$ such that
$(s,a) \in Z$ and $\dim Z_s = k$. By cell decomposition in o-minimal structures, there are an $\Cal L_A$-definable cell $X\subseteq M^{l+k}$ and a continuous $\Cal L_A$-definable function $h : X \to M^{n}$ such that
\begin{itemize}
\item $\{(x, h(x,y)): (x,y)\in X\}\subseteq Z$
\item $(s,a_1,\dots,a_k)\in X$ and $h(s,a_1,\dots,a_k)=a$,
\item $X_y$ is an open cell and $h(y,-)$ is injective for each $y\in \pi_l(X)$.
\end{itemize}
In particular, $(X_y)_{y\in\pi_l(X)}$ is a uniform family of supercones with closure $cl(X)$. Thus
\[
h\left(\bigcup_{g \in P^l\cap \pi_l(X)}\{ g \} \times X_g\right)
\]
is a $k$-cone containing $a$.
\end{proof}




\begin{lem}\label{lem:2n0} Let $C$ be an $A$-definable $0$-cone in $M^n$ and $f: C\to M$ be $A$-definable. Then there is a finite collection $\Cal C$ of $A$-definable $0$-cones whose union is $C$ and such that $f$ is fiber $\Cal L$-definable with respect to each cone in $\Cal C$.
\end{lem}
\begin{proof}
Let $S$ be $A$-definable small and $h: Z \sub M^m\to M^n$ be $\Cal L_A$-definable and continuous such that $h(S)=C$. We may assume that $S\sub P^l$, for some $l$. Indeed, since $S$ is $P$-bound over $A$, one can easily see that $S$ is a finite union of sets $\sigma(S')$, where $S'\sub P^l$ is $A$-definable and $\sigma: W\to M^m$ is an $\cal L_A$-definable map. So $C$ is a finite union of $0$-cones of the form $h \circ \sigma(S')$.

Now, by Corollary \ref{cor:pprop1} there are $k,t\in \N$ and, for $i=1,\dots,k$, an $\Cal L_{A}$-definable continuous function $F_i : Z_i  \subseteq P^t \times M^{l} \to M$ with $Z_i$ a cell,  and $s \in P^t$, such that for all $g \in S$ there is $i\in \{1,\dots,k\}$ with $(f\circ h)(g)=F_i(s,g)$. Now set
\[
S_i := \{ (s,g) \in P^t \times S \ : \ (s,g) \in Z_i, (f\circ h)(g)=F_i(s,g)\}.
\]
Set $\tau : M^t \times Z \to M^n$ to map $(x,y)$ to $h(y)$. Then $\tau(S_i)$ is an $A$-definable $0$-cone and $f$ is fiber $\Cal L_A$-definable with respect to $\tau(S_i)$. Moreover, $C = \bigcup_{i=1}^k \tau(S_i)$.
\end{proof}

Our next goal is to prove Lemmas \ref{cone+1} and \ref{largefibers} below, which will be used in the proof of the Structure Theorem (1)$_n$, Cases I and II, respectively. First, a lemma about shells.

\begin{lemma}\label{shell-canonical}
Let $\cal J\sub M^{m+k}$ be an $A$-definable uniform family of supercones with an  $\cal L_A$-definable shell $V$. Assume that $Z\sub M^{m+k}$ is an $\cal L_A$-definable cell containing $\cal J$. Then there are disjoint $A$-definable uniform families of supercones $\cal J_1, \dots, \cal J_n$ such that
$$\cal J=\cal J_1 \cup \dots\cup\cal J_n,$$
and each $\cal J_i$ has an $\cal L_A$-definable shell $V_i \sub V\cap Z$.
\end{lemma}
\begin{proof} First observe that for every $g\in \pi(\cal J)$, $V_g\sub Z_g$. Indeed,
$cl(V_g)=cl(\cal J_g)\sub cl(Z_g)$. Since $V_g$ is an open cell, and $Z_g$ is a cell too, this implies that $V_g\sub Z_g$.

Now let
$$D=\{g\in M^k : (V\cap Z)_g \text{ is an open cell}\}.$$
This set is $\cal L_A$-definable. Moreover, since for every $g\in \pi(\cal J)$, $(V\cap Z)_g=V_g$, we obtain  $\pi(\cal J)\sub D$. Let
$$D=D_1\cup\dots\cup D_n$$
be a partition of $D$ into $\cal L_A$-definable cells, and, for each $i$,
$$\cal J_i=\cal J \cap (D_i\times M^k)$$
and
$$Z_i=(V\cap Z)\cap (D_i\times M^k).$$
Since both $V, Z$ are cells and $D_i\sub D$, it is not hard to see that each $Z_i$ is a cell.
It clearly also contains $\cal J_i$. Finally, for every  $g\in D_i$ and  $0<j\le k$, we have
$$cl(\pi_{m+j}(\cal J_i)_g)= cl(\pi_{m+j}(V)_g)= cl(\pi_{m+j}(V\cap Z)_g)=  cl(\pi_{m+j}(Z_i)_g),$$
showing that $Z_i$ is a shell for $\cal J_i$.
\end{proof}



We now prove that a suitable family of large subsets of $M$ ranging over a $k$-cone gives rise to a $k+1$-cone.

\begin{lemma}\label{cone+1}
Let $C\subseteq M^n$ be an $A$-definable $k$-cone, let $\{X_a\}_{a\in C}$ be an $A$-definable family of subsets of $M$. Assume that $h_1, h_2: C\to M\cup\{\pm\infty\}$ are fiber $\cal L_A$-definable with respect to $C$, and such that for all $a\in C$, $X_a$ is contained in $(h_1(a), h_2(a))$ and it is co-small in it. Then $\bigcup_{a\in C} \{a\}\times X_a$ is a finite disjoint union of $A$-definable $k+1$-cones. 
\end{lemma}
\begin{proof}
Suppose that $C=h(\Cal J)$ for some uniform family $\cal J=\{J_g\}_{g\in S}$ of supercones in $M^k$ with shell $V$ and $\Cal L_{A}$-definable continuous $h: U \subseteq M^{m+k}\to M^n$, where $U$ is a cell containing $\cal J$. By the assumption on $h_1$ and $h_2$, there are $\Cal L_A$-definable continuous functions $H_1,H_2: Z \subseteq M^{m+k} \to M$, where $Z$ is a cell containing $\cal J$, such that for every $g \in S$ and $t\in J_g$,
\[
H_1(g,t)=h_1(h(g,t)) \hbox{ and } H_2(g,t)=h_2(h(g,t)).
\]
By Lemma \ref{shell-canonical}, we may assume that $V\sub U\cap Z$.
Now set
\[
V':=\{ (g,t,x) \ : \ (g,t) \in V, \ H_1(g,t) < x < H_2(g,t)\}
\]
and
\[
J'_g:=\bigcup_{t\in J_g} \{t\}\times X_{h(g, t)}.
\]
It is easy to check that $\{J'_g\}_{g\in S}$ is a uniform family of supercones in $M^{k+1}$ with closure $cl(V')$. Let $\tau: V\times M \to M^{n+1}$ map $((g,t),x)$ to $(h(g, t), x)$. For each $g\in S$ the function $\tau(g, -)$ is injective, because so is $h(g, -)$. Thus
\[
\bigcup_{a\in C} \{a\}\times X_a =\tau\left(\bigcup_{g\in S} \{g\}\times J'_g\right)
\]
is a $k+1$-cone.
\end{proof}

The proof of the Structure Theorem will run in parallel with its own uniform version (see Theorem \ref{str-thm-sets}(3) below), which prompts the following definition.

\begin{defn}[Uniform families of cones]\label{def-unifcones} Let $\Cal C:=\{C_t\}_{t\in X\subseteq M^m}$ be a definable family of $k$-cones in $M^n$. We call $\Cal C$ \emph{uniform} if there are
\begin{itemize}
\item an $\Cal L$-definable continuous function $h: Z \subseteq M^{m+l+k} \to M^n$,
\item a definable family  $\{S_t\}_{t\in X}$ of small subsets of $M^l$,
\item a  uniform definable family of supercones $Y= \{ Y_{t,g}\}_{t\in X, g \in S_t}$ in $M^k$
\end{itemize}
such that  $Y\sub Z$ and
\begin{itemize}
\item[(i)] $h(t,g,-) : Z_{t,g}\subseteq M^k \to M^n$ is injective for each $g\in S_t$,
\item[(ii)] $C_t=h\left(\{t\} \times \left(\bigcup_{g\in S_t} \{g\}\times Y_{t,g}\right)\right)$.
\end{itemize}
Abusing terminology, we call $\cal C$ \emph{$A$-definable} if it is an $A$-definable family of sets, $h$ is $\cal L_A$-definable, and  $\{S_t\}_{t\in X}$ and $\{Y_{t, g}\}_{t\in X, g \in S_t}$ are $A$-definable.
\end{defn}

We now prove that the union of a small uniform family of $k$-cones under a suitable map results again in a $k$-cone.

\begin{lemma}\label{largefibers}
Let $\{C_t\}_{t\in K}$ be an $A$-definable uniform family of $k$-cones in $M^n$, with $K\sub M^m$ small, and let $\tau:W \subseteq M^{m+n}\to M^p$ be an $\cal L_A$-definable continuous map such that for each $t\in K$
\begin{itemize}
\item $\{t\}\times C_t\subseteq W$,
\item $\tau(t, -):M^n\to M^p$ is injective.
\end{itemize}
Then $\tau\left(\bigcup_{t\in K} \{t\}\times C_t\right)$ is an $A$-definable $k$-cone in $M^p$.
\end{lemma}
\begin{proof} Let $h: Z \subseteq M^{m+l+k} \to M^n$ be an $\Cal L_A$-definable continuous function,   $\{S_t\}_{t\in K}$   an $A$-definable family of small subsets of $M^l$, and  $\{ Y_{t,g}\}_{t\in K, g \in S_t}$ an $A$-definable family of supercones that witness that $\{C_t\}_{t\in K}$ is a uniform family of $k$-cones. Let $\sigma: Z \subseteq M^{m+l+k}\to M^{p}$ be defined by $\sigma(t, g, a):=\tau(t, h(t,g,a))$. We see directly that $\sigma(t,g,-)$ is injective, since $\tau(t,-)$ and $h(t,g,-)$ are injective. Note also that $\sigma$ is $\Cal L_A$-definable and continuous, since both $h$ and $\tau$ are. Set
\[
S :=\bigcup_{t\in K} \{t\}\times S_t.
\]
It is then straightforward to check that
\[
\tau\left(\bigcup_{t\in K} \{t\}\times C_t\right)=\sigma\left(\bigcup_{(t,g)\in S} \{(t, g)\}\times  Y_{t,g}\right)
\]
is the desired $k$-cone.
\end{proof}

The following lemma will be used in the last step of the proof of the Structure Theorem, $(1)_{n} \,\Rarr\, (3)_{n}$. It follows easily from Definition \ref{def-uniform} and the next observations. Let $X\sub M^{m+n}$ be a set. Then for every $0<j\le n$   and $g\in \pi_{m}(X)$, we have
$$\pi_{m+j}(X)_g=\pi_j(X_g).$$
Let $X, Y\sub M^n$ and $0<j\le n$. Then
$$cl(X)=cl(Y) \,\Rarr\,\ cl(\pi_j(X))=cl(\pi_j(Y)).$$
Indeed, $\pi_j(X)\sub \pi_j(cl(Y))\sub cl(\pi_j(Y))$.

\begin{lemma}\label{uniform family} Let  $U\sub M^{m+l+k}$ be an $A$-definable cell. Let
$$\cal K=\{J_{t, g}\}_{t\in Y,\, g\in S_t}$$
be an $A$-definable family of supercones $J_{t,g}\sub M^k$, where $Y\sub M^m$ and $S_t\sub M^l$. Assume  that for every $0<j\le k$, $t\in Y$ and $g\in S_t$,
\begin{equation}
  cl(J_{t,g})=cl(U_{t,g}).\label{eq-unif}
\end{equation}
Then $U$ is a shell for $\cal K$. In particular,  $\cal K$ is an $A$-definable  uniform family of supercones.
\end{lemma}
\begin{proof}  For every $t\in Y$ and $g\in S_t$, we have
$$cl(\cal J_{t,g})=cl(U_{t,g})\,\Rarr\, cl(\pi_j(\cal J_{t,g}))=cl(\pi_j(U_{t,g}))\,\Rarr$$ $$\Rarr\, cl(\pi_{m+l+j}(\cal J)_{t,g})=cl(\pi_{m+l+j}(U)_{t,g}),$$
as required.
\end{proof}
 We finally include two lemmas that will be useful in the discussion of `large dimension' in Section \ref{sec-large dimension} below.

\begin{lemma}\label{low1}
Let $J\sub M^n$, $n>0$, be a supercone and $X\sub M^n$   a low set. Then  $J\setminus X$ contains a supercone.
\end{lemma}
\begin{proof} Easy, following the definitions, by induction on $n$.
\end{proof}

\begin{lemma}\label{coneunion1}
Let $J\sub M^n$ be a supercone and $\{X_s\}_{s\in S}$ a small definable family of subsets of $M^n$ such that $J= \bigcup_{s\in S} X_s$. Then some $X_s$ contains a supercone in $M^n$.
\end{lemma}
\begin{proof}
By induction on $n$. If $n=0$, it is obvious. If $n>0$, for every $s\in S$, let
$$Y_s:= \{t\in \pi(X_s): \text{ the fiber } (X_s)_t \text{ is large}\}.$$
By Remark \ref{definability}(a), $\{Y_s\}_{s\in S}$ is a definable family of sets. By Corollary \ref{smallunion}, we have $\pi(J)= \bigcup_{s\in S} Y_s$. By Inductive Hypothesis, some $Y_s$ contains a supercone $K$. Since for every $t\in K$, $(X_s)_t$ is large, Remark \ref{definability}(b) provides us with definable functions $h_1, h_2: M^{n-1}\to M\cup\{\pm\infty\}$ such that for every $t\in \pi(X_s)$, $(X_s)_t$ is co-small in $(h_1(t), h_2(t))$. By Corollary \ref{cor:corp1}, there are finitely many low sets in $M^{n-1}$ off whose union $h_1, h_2$ are both $\cal L$-definable and continuous. Hence, by repeated use of Lemma \ref{low1}, we obtain a supercone $K'$ contained in $K$ on which $h_1, h_2$ are both $\cal L$-definable. Therefore, the set
$$\bigcup_{t\in K'} \{t\} \times (X_s)_t\cap (h_1(t), h_2(t))$$
is a supercone contained in $X_s$.
\end{proof}

\subsection{\cal L-definable functions on supercones}

The goal  of this section (Proposition \ref{k}(1) below) is to show that a supercone from $M^m$ cannot be `embedded' into $M^n$, for $n<m$. This will make meaningful the notion of `large dimension'  we introduce in  Section  \ref{sec-large dimension}.

\begin{lem}\label{open-supercone}
 Let $J\sub M^n$ be an $A$-definable supercone and $S\sub cl(J)$ an open $\Cal L_A$-definable cell.  Then $S\cap J$ is an $A$-definable supercone with closure $cl(S)$.
\end{lem}
\begin{proof} We work by induction on $n$. For $n=0$ it is obvious. Assume we know the statement for subsets of $M^k$, $k<n$, and let $J\sub M^{n}$ be a supercone and $S\sub cl(J)$ be an open $\Cal L_A$-definable cell. Since $\pi(S)\sub \pi(cl(J))\sub cl(\pi(J))$, the inductive hypothesis gives that $\pi(S)\cap \pi(J)$ is an $A$-definable supercone $K\sub M^{n-1}$ with closure $cl(\pi(S))$. Since for every $t\in K$, $J_t$ is co-small in $cl(J)_t$, we have that $(S\cap J)_t=S_t\cap J_t$ is co-small in $S_t$. Hence $S\cap J=\bigcup_{t\in K} \{t\} \times (S\cap J)_t$ is a supercone with closure $cl(S)$.
\end{proof}

\begin{lemma}\label{U-K}
Let $K\sub M^{n+1}$ be a supercone. Then $cl(K)\sm K$ is a finite union of sets of the form
$$\bigcup_{g\in P^m}h(g, Z_{g}),$$
where $Z\sub P^m\times M^n$ is definable, $h:M^{m+n}\to M^{n+1}$ is $\cal L$-definable and each $h(g, -):Z_g\to M^{n+1}$ is injective.
\end{lemma}
\begin{proof}
By induction on $n$. Denote $U=cl(K)$. For $n=0$, this is clear since $U\sm K$ is a small set and can be written as $h(P^m)$ with $h$ as above. Now assume we know the statement for $k<n$, let $K\sub M^{n+1}$ be as above. We have:
\begin{equation}
U\sm K=\left(\bigcup_{t\in \pi(K)}\{t\}\times (U_t\sm K_t)\right)\cup\left(\bigcup_{t\in \pi(U)\sm \pi(K)}\{t\}\times U_t\right).\label{eqU-K}
\end{equation}
By inductive hypothesis the second part is a finite union of sets of the form
$$T=\bigcup_{t\in X}\{t\}\times U_t,$$
where $X=\bigcup_{g\in P^m} h(g, Z_g)$, for suitable $h$. Observe  that then
$$T=\bigcup_{g\in P^m} h'(g, W_g),$$
where $W_g=\bigcup_{v\in Z_g} \{v\}\times U_{h(g, u)}$ and $h'(g, v,u)=(h(g, v), u)$, as required.

The first part of the union in (\ref{eqU-K}) is of the right form, as it follows immediately by applying Lemma \ref{philippcor}.
\end{proof}

Before proving Proposition \ref{k}, we  illustrate it with an example.

\begin{example} Consider the function $f:M^2\to M$ with $f(x_1, x_2)=x_1+x_2$. Let $J_1=M\sm P$ and for all $t\in J_1$, $J_t= J_1\cap (t, \infty)$. Let $J=\bigcup_{t\in J_1} \{t\}\times J_t$. 
We will show that $f_{\res J}$ is not injective.
 The proof is inspired by an example in \cite[page 5]{beg}. Assume towards a contradiction that  $f_{\res J}$ is injective.  Pick any two distinct $t_0> t\in J$. Since $f_{\res J}$ is injective, for every $b\in t_0+J_{t_0}$, we have $b\not\in t+J_t$.  But $b\in t+cl(J_t)$, so $b\in t+P$. Since this holds for every $b\in t_0+ J_{t_0}$, we have that $t_0+J_{t_0}\sub t+P$, which is a contradiction, since a large set cannot be contained in a small one.
\end{example}


\begin{prop}\label{k} Let $f:M^m\to M^n$ be an $\cal L$-definable function and  $J\sub M^m$ a supercone, such that $f_{\res J}$ is injective. Then
\begin{enumerate}
\item $m\le n$.

\item there is an $\cal L$-definable $X\sub cl(J)$ such that $\dim(cl(J)\sm X)<m$ and $f_{\res X}$ is finite-to-one. Namely, $X=X_f\cap cl(J)$, with notation from Fact \ref{fin2}.
    \item If  $K\sub M^{n}$ is another supercone and $f:cl(J)\to cl(K)$ is injective, then $f(J)\cap K\ne \emptyset$.
\end{enumerate}
In particular,  by (2), there is an open $\cal L$-definable $X\sub cl(J)$ such that $f_{\res X}$ is injective.
\end{prop}


\begin{proof} The last clause follows from Fact \ref{fin1}.\vskip.2cm

We write (1)$_m$ - (3)$_m$ for the above statements, and prove them simultaneously by induction on $m$. 
Statement $(1)_1$ is clear. Let $m\ge 1$.

\vskip.2cm \noindent    $\mathbf{(1)_{m}\, \Rarr (2)_{m}.}$ Denote
$$X_f=\{a\in cl(J): f^{-1}(f(a)) \text{ is finite}\}.$$
We claim that $\dim(cl(J)\sm X_f)<m$. Assume not. Let $I\sub cl(J)\sm X_f$ be an open box. By Lemma \ref{open-supercone}, $I\cap J$ contains a supercone $K\sub M^{m}$. By Fact \ref{fin2}, $f(I)$ has dimension $l<m$. In particular, $f(I)$ is in definable bijection with a subset of $M^l$ via the restriction of an $\cal L$-definable map $h:M^n\to M^l$. Consider now $g=h\circ f:M^{m}\to M^l$. Then $g$ is $\cal L$-definable and  injective on $K$. We have contradicted $(1)_{m}$.

\vskip.2cm \noindent $\mathbf{(1)_{m}\Rarr (3)_{m}.}$  Let $K\sub M^{m}$ be a supercone and assume that $f:cl(J)\to cl(K)$ is injective.  Suppose now for a contradiction that $f(J)\sub cl(K)\sm K$. By Lemma \ref{U-K} and Corollary \ref{cor-union}, $cl(K)\sm K$ is contained in the union of a small definable family of sets each of the form $h(g, Z_g)$ (for finitely many   $h$'s), with each $Z_g\sub M^{m-1}$ and each $h(g, -):Z_g\to M^{m}$ being $\cal L$-definable and injective. In particular,  $J$ is the union of a small definable family of sets of the form $f^{-1}h(g, Z_g)\cap J$. By Lemma \ref{coneunion1}, one of those sets must contain a supercone $L\sub M^n$. By Lemma \ref{open-supercone}, $T:=\text{interior of } cl(L))\cap J$ is a supercone in $M^m$. But then the map $F=h(g, -)^{-1} \circ f : cl(T)\to M^{m-1}$ is an $\cal L$-definable map that is injective on $T$, contradicting (1)$_{m}$.

\vskip.2cm \noindent  $\mathbf{ (2)_m \,\&\, (3)_m \Rarr  (1) _{m+1}.}$ Let $f:M^{m+1}\to M^n$ be an $\cal L$-definable function and $J\sub M^{m+1}$ a supercone with closure $V$ such that $f_{\res J}$ is injective. Assume towards a contradiction that $m\ge n$.
Let $J_1=\pi_1(J)$ be the projection of $J$ onto the first coordinate, and $V_1=\pi_1(V)$.
By $(2)_m$, for every $t\in J_1$, there is an open box  $X_t\sub Y_t$ on which $f(t, -)$ is injective. By cell decomposition in o-minimal structures, and since $J_1$ is dense in $V_1$, there is an open cell $U\sub V$, such that for every $t\in \pi_1(U)$, $f(t, -)$ is injective on $U_t$.
By Lemma \ref{open-supercone}, $U\cap J$ is a supercone with closure $cl(U)$. We may thus replace $J$ by $U\cap J$, and $V$ by $cl(U)$,  and assume from now on that
for every $t \in V_1$, $f(t,-)$ is injective on $V_t$. \vskip.2cm

\noindent\textbf{Claim 1.} \emph{There is an  open interval $I_1\sub V_1$ and an open box $I\sub M^n$, such that for every $t\in I_1$, $I\sub f(t, V_t)$.}

\begin{proof}[Proof of Claim 1.] Since  for every $t\in V_1$, $f(t, -)$ is injective on $V_t$, it follows that the dimension of the $\cal L$-definable set $$Z =\bigcup_{t\in V_1} \{t\}\times f(t, V_t)$$ is $n+1$. By cell decomposition, there is an open interval $I_1\sub V_1$ and an open box $I\sub M^n$ 
such that $I_1\times  I\sub Z$. In particular, for all $t\in I_1$, $I\sub f(t, V_t)$.
\end{proof}

By Claim 1, we can pick two distinct $t_0, t\in J_1$ such that
$$I\sub f(t_0, V_{t_0})\cap f(t, V_t)$$
has dimension $n$. Since $f_{\res J}$  is injective, for any $b\in I\cap f(t_0, J_{t_0})$, we have $b\not\in f(t, J_t)$, and hence $b\in f(t, V_t\sm J_t)$. Since this holds for every $b\in I\cap f(t_0, J_{t_0})$, we have that
$$I\cap f(t_0, J_{t_0})\sub f(t, V_t\sm J_t).$$

\vskip.1cm

\noindent\textbf{Claim 2.} \emph{There is a supercone $T\sub V_{t_0}$ such that $f(t_0, T)\sub I\cap f(t_0, J_{t_0})$.}
\begin{proof}
Denote $f_{t_0}(-)=f(t_0, -)$. So $f_{t_0}$ is injective on $V_{t_0}$. Since $I\sub f(t_0, V_{t_0})$, we have $f_{t_0}^{-1}(I)\sub V_{t_0}$. Let $I'\sub f_{t_0}^{-1}(I)$ be an open cell. By Lemma \ref{open-supercone}, $T:=I'\cap J_{t_0}$ is a supercone, as required.
\end{proof}

We conclude that the  map $f(t, -)^{-1} \circ f(t_0, -):V_{t_0}\to V_t$ is an injective $\cal L$-definable map that maps $T$ into $V_{t}\sm J_{t}$, contradicting (3)$_m$.
\end{proof}


We show with an example that the assumption on $J$ being a supercone (and not just satisfying  $\dim(cl(J))=m$) is necessary.

\begin{example}\label{exa-lou1} Let $f$ be the function from Example \ref{exa-fiber}. The usual projection map $\pi:\R^2\to\R$ is injective on $Graph(f)$ but of course not injective on any open subset of $cl(Graph(f))=\R^2$.
\end{example}

The next definition and corollary will be useful when we discuss the notion of large dimension in Section  \ref{sec-large dimension}.
\begin{defn}
Let $f:M^k\to M^n$ be an $\cal L$-definable map, $J\sub M^k$ a supercone and $X\sub M^n$ a definable set. We say that
\begin{itemize}
  \item $f$ is a \emph{strong embedding of $J$ into $X$} if $f$ is injective and $f(J)\sub X$.

  \item $f$ is a \emph{weak embedding of $J$ into $X$} if $f_{\res J}$ is injective and $f(J)\sub X$.
\end{itemize}
\end{defn}


\begin{cor}\label{embed} Let $X\sub M^n$ be a definable set. The following are equivalent:
\begin{enumerate}
  \item there is a weak embedding of a supercone $J\sub M^k$ into $X$.

 \item there is a supercone $K\sub M^k$ and an \cal L-definable $f:M^k\to M^n$, injective on $cl(K)$, with $f(K)\sub X$.

  \item there is a strong embedding of a supercone $L\sub M^k$ into $X$.
 \end{enumerate}
\end{cor}
\begin{proof} (3)\Rarr (1) is obvious.

(1)\Rarr (2). Let $f:M^k\to M^n$ be an $\cal L$-definable map, injective on $J$, with $f(J)\sub X$. By Proposition \ref{k}, there is an open definable $S\sub cl(J)$ such that $f_{\res cl(S)}$ is injective. By Lemma \ref{open-supercone}, $J\cap S$ contains a supercone $K$.

(2)\Rarr (3). Let $S\sub cl(K)$ be open so that $f_{\res S}$ can be extended to an injective $\cal L$-definable map $F:M^k\to M^n$. By Lemma \ref{open-supercone} again, $S\cap K$ contains a supercone $L$.
\end{proof}

\subsection{Large dimension}\label{sec-large dimension}

We introduce an invariant for every definable set $X$ which tends to measure `how large' $X$ is. This invariant will be used in the inductive proof of the Structure Theorem in Section \ref{sec-str-thms}. 

\begin{defn}\label{def-large}
Let $X\sub M^n$ be definable. If $X\ne \emptyset$, the \emph{large dimension} of $X$ is the maximum $k\in \bb N$ such that $X$ contains a $k$-cone. Equivalently, it is the maximum  $k\in \bb N$ such that there is a strong embedding of a supercone $J\sub M^k$ into $X$. We also define the large dimension of the empty set to be $-\infty$. We denote the large dimension of $X$ by  $\ldim(X)$.
\end{defn}
Clearly, the large dimension of a subset of $M^n$ is bounded by $n$. In view of Corollary \ref{embed}, the large dimension of $X$ is the maximum  $k\in \bb N$ such that there is a weak embedding of a supercone $J\sub M^k$ into $X$.
In Section \ref{sec-equaldim}, we will prove that the large dimension equals the `$\scl$-dimension' arising from a relevant pregeometry in \cite{beg}. Here we establish some of its basic properties. The first lemma is obvious.

\begin{lemma}
  For every definable $X, Y\sub M^n$, if $X\sub Y$, then $\ldim(X)\le \ldim(Y)$.
\end{lemma}

\begin{lemma}\label{coneunion2}
Let  $\{Z_s\}_{s\in S}$ be a  small definable family of sets. Then
$$\ldim\left( \bigcup_{s\in S} Z_s\right)=\max \ldim Z_s.$$
\end{lemma}
\begin{proof}  ($\le$). Assume $f:M^n\to M^m$ is an $\cal L$-definable injective map, $J\sub M^n$ is a supercone, and $f(J)\sub  \bigcup_{s\in S} Z_s$. We show that for some $s\in S$, $\ldim(Z_s)\ge n$. For every $s\in S$, let $X_s:=f^{-1}(Z_s).$ Then $\{X_s\cap J\}_{s\in S}$ is a definable family of subsets of $M^n$ that cover $J$, and by Lemma \ref{coneunion1}, one of them must contain a supercone $K\sub M^n$. Since $f(K)\sub Z_s$, we have that $\ldim(Z_s)\ge n$.

($\ge$). This is clear.
\end{proof}

In particular, we obtain the following standard property that holds for any good notion of dimension.

\begin{cor}\label{union}   Let $X_1, \dots, X_l$ be definable sets. Then
$$\ldim(X_1\cup\dots\cup X_l)=\max\{\ldim(X_1), \dots, \ldim(X_l)\}.$$
\end{cor}




About supercones and cones we have:

\begin{cor}\label{conedim}
If $C\sub M^n$ is a $k$-cone, then $\ldim(C)=k$.
 \end{cor}
 \begin{proof}
By Lemma \ref{coneunion2} and the definition of a cone it suffices to show that every supercone in $M^k$ has large dimension $k$. But this is clear.
 \end{proof}

\begin{cor}\label{clJ-J}
Let $n>0$ and $J\sub M^{n}$ be a supercone. Then $\ldim(cl(J)\sm J)<n$.
\end{cor}
\begin{proof}
Immediate from Proposition \ref{k}(3) and the definitions.
\end{proof}

\begin{lemma}\label{ldimproj}
Let $X\sub M^{n+1}$ be a definable set, such that for every $t\in \pi(X)$, $X_t$ is small. Then $\ldim(X)=\ldim(\pi(X))$.
\end{lemma}
\begin{proof} Let $U_i$, $S_i$, $h_i$ and $Z_{ig}$ be as in Lemma \ref{philippcor}. In particular,
\begin{equation}
U_i=h_i\left(\bigcup_{g\in S_i} \{g\}\times Z_{ig}\right).\label{eq-ldimproj}
\end{equation}

($\ge$).  By Lemma \ref{philippcor}(3), we have $\pi(X)=\bigcup_{i,g} Z_{ig}$. By Lemma \ref{coneunion2}, for some $i, g$, we have $\ldim(Z_{ig})=\ldim(\pi(X))$. By Equation (\ref{eq-ldimproj}) and Lemma \ref{philippcor}(1), we obtain
$$\ldim(Z_{ig})\le \ldim(U_i)\le \ldim(X).$$

($\le$). By Corollary \ref{union},  $\ldim(X)= \max_i \ldim(U_i)$. By Equation (\ref{eq-ldimproj}),  Lemma \ref{philippcor}(2) and Lemma \ref{coneunion2}, for every $i$, $\ldim(U_i)=\max_g \ldim(Z_{ig})$. But $Z_{ig}\sub \pi(X)$, so $\ldim(X)\le \ldim(\pi(X))$.
\end{proof}

\begin{cor}\label{smallprojection}
Let $X\sub M^n$ be a definable set. Then $\ldim(X)=0$ if and only if $X$ is small.
\end{cor}
\begin{proof}
Right-to-left is immediate from the definitions of a small set and large dimension. For the left-to-right, we use induction on $n$. If $n=1$, the statement is clear by Lemma \ref{lem:beg1}. Assume we know the statement for all $l\le n$ and let $X\sub M^{n+1}$. \vskip.2cm

\noindent\textbf{Claim.} \emph{The projection of $X$ onto any of its coordinates is small.}
\begin{proof}[Proof of Claim]
Without loss of generality we may just prove that the projection $\pi(X)$ onto the first $n$ coordinates is small. Since $\ldim(X)=0$, using Lemma \ref{lem:beg1}, we see that for every $t\in \pi(X)$, $X_t$ is small. By Lemma \ref{ldimproj}, $\ldim(\pi(X))=\ldim(X)=0$. By Inductive Hypothesis, $\pi(X)$ is small.
\end{proof}

Since $X$ is contained in the product of its coordinate projections, it is again small.
\end{proof}

In Definition \ref{def-low}, we introduced low sets. We are now able to determine their large dimension.

\begin{lemma}\label{nonlow}
Let $X\sub M^n$ be a low definable set. Then $\ldim(X)=n-1$.
\end{lemma}
\begin{proof}
By Remark \ref{rem-supercone}(3) and Corollary \ref{conedim}.
\end{proof}


\begin{remark}
We observe that the converse of Lemma \ref{nonlow} does not hold, even if we allow finite unions of low definable sets.
For example, let $X:= (M\setminus P) \times P$. One can see that $X$ is a 1-cone. Suppose $X$ is the finite union of low sets. Then the image
of $X$ under at least one of the coordinate projections has interior. But the images of $X$ under the two coordinate projections are $M\setminus P$ and $P$. Neither of these two sets has nonempty interior.
\end{remark}



\section{Structure theorem}\label{sec-str-thms}
We are now ready to prove the main result of this paper, which consists of statements (1) and (2) below. The proof runs by simultaneous induction along with statement  (3). The latter is a uniform version of (1).

\begin{theorem}[Structure Theorem]\label{str-thm-sets}$ $

\begin{enumerate}
  \item  Let $X\sub M^n$ be an $A$-definable set. Then $X$ is a finite union of $A$-definable cones. 

  \item  Let $f:X\sub M^n\to M$ be an $A$-definable function. Then there is a finite collection $\cal C$ of $A$-definable cones whose union is $X$ and such that $f$ is fiber $\cal L_A$-definable with respect to each $C\in \cal C$.
\item Let $\{X_t\}_{t \in M^m}$ be an $A$-definable family of subsets of $M^n$. Then there is $p\in \N$ and for every $i\in\{1,\dots,p\}$,
\begin{itemize}
\item an $A$-definable subset $Y_i\subseteq M^m$,
\item $k_i \in \N$,
\item an $A$-definable uniform family of $k_i$-cones $\{C^i_t\}_{t\in Y_i}$,
\end{itemize}
such that for all $t \in M^m$
\[
X_t = \bigcup \left\{C^i_t:  t \in Y_i\right\}.
\]
\end{enumerate}
\end{theorem}

\begin{proof}
We write (1)$_n$ - (3)$_n$ for the above statements. We will now show by induction on $n$ that (1$)_n$ - (3$)_n$ hold. Statements (1)$_0$ - (3)$_0$ are trivial. Suppose now that $n>0$ and (1)$_l$ - (3)$_l$ hold for every $l< n$. It is left to show  (1)$_{n}$ - (3)$_{n}$.

\vskip.2cm \noindent $\mathbf{(1)_n}.$ Let $X\sub M^n$. By Remark \ref{definability}(b), we may assume that there are $A$-definable $h_1, h_2:M^{n-1}\to M\cup\{\pm\infty\}$ such that for every $a\in \pi(X)$, $X_a$ is contained in $(h_1(a), h_2(a))$, and it is either small in it for all $a\in \pi(X)$, or co-small in it for all $a\in \pi(X)$. We handle the two cases separately.\\

\noindent\textbf{Case I:} For every $a\in \pi(X)$, $X_a$ is co-small in $(h_1(a), h_2(a))$. \newline
By  (2)$_{n-1}$, we may assume that $\pi(X)$ is an $A$-definable cone, such that $h_1, h_2$ are fiber $\cal L_A$-definable with respect to it.  By Lemma \ref{cone+1}, $X$ is a finite union of $A$-definable cones.\\

\noindent\textbf{Case II:} For every $a\in \pi(X)$, $X_a$ is small in $(h_1(a), h_2(a))$. \newline
By Lemma \ref{philippcor}, we may assume that there are an $\cal L_A$-definable continuous function $h:Y \subseteq M^{m+n-1}\to M^n$, and $A$-definable small set $S\sub M^m$, and an $A$-definable family $\{Z_g\}_{g\in S}$ with $Z_g\sub \pi(X)$ such that
\begin{itemize}
\item $X=h\left(\bigcup_{g\in S} \{g\}\times Z_g\right)$, and
\item for every $g\in S$, $h(g, -):M^{n-1} \to M^{n}$ is injective.
\end{itemize}
By ($3)_{n-1}$, there is $p\in \N$ and for every $i\in\{1,\dots,p\}$,
\begin{itemize}
\item an $A$-definable subset $Y_i\subseteq S$,
\item $k_i \in \N$,
\item an $A$-definable uniform family of $k_i$-cones $\{C^i_g\}_{g\in Y_i}$,
\end{itemize}
such that for all $g \in S$,
\[
Z_g = \bigcup \left\{C_g^j : g\in Y_j\right\}.
\]
By Lemma \ref{largefibers}, we have that for each $j\in \{1,\dots,p\}$,
\[
h\left(\bigcup_{g\in Y_j} \{g\}\times C^j_g\right)
\]
is an $A$-definable $k_j$-cone. Thus $X$ is a finite union of $A$-definable cones.

\vskip.2cm \noindent   $\mathbf{(1)_{n} \,\Rarr\, (2)_{n}.}$ Let $f:X\sub M^{n}\to M$ be an $A$-definable function. We prove (2)$_n$ by sub-induction on $\ldim(X)$. Suppose first that $\ldim(X)=0$. By (1$)_n$ we can assume that $X$ is a $0$-cone. By Lemma \ref{lem:2n0} we can find a finite collection $\cal C$ of $A$-definable cones whose union is $X$ and such that $f$ is fiber $\cal L_A$-definable with respect to each $C\in \cal C$. So we can now assume that $\ldim(X)=k>0$ and ($2)_n$ holds for all definable functions whose domain has $\ldim<k$. By (1)$_{n}$, we may assume $X\sub M^{n}$ is an $A$-definable $k$-cone, say $X=h(\cal J)$. Let $S=\pi(\cal J)$. We now apply Corollary \ref{philipp2uniform} to $f\circ h : \bigcup_{g\in S} \{g\}\times J_g \to M$ to get  $p,t\in \N$ and for every $i\in \{1,\dots,p\}$
 \begin{itemize}
 \item an $A$-definable family $\{X_{g}\sub M^k\}_{g\in S}$ with $\ldim(X_g)<k$,
 \item an $\Cal L_A$-definable continuous function  $f^i : Z_i \subseteq M^{l+t+k} \to M$
\end{itemize}
 such that for every $g\in S$ there is $u \in P^{t}$ such that
\begin{itemize}
\item[(A)] for all $a\in J_g \sm X_g$ there is $i \in \{1,\dots,p\}$ such that $(f\circ h)(g,a) = f^i(g,u,a)$.
\end{itemize}
We denote the set of all pairs $(g,u)\in S\times P^t$ that satisfy (A) by $K$. For each $i\in \{1,\dots,p\}$ we define for $(g,u) \in K$,
\[
B^i_{g,u} = \{ a \in J_g \sm X_g \ : \ (f\circ h)(g,a)=f^i(g,u,a)\}.
\]
Note that for $g\in S$,
\[
\bigcup_{u \in K_g} \left( J_g\setminus \bigcup_i B^i_{g,u} \right)\subseteq X_g.
\]
Therefore
\[
\ldim \left(\bigcup_{u \in K_g} \left(J_g\setminus \bigcup_i B^i_{g,u}\right)\right)<k.
\]
Since $h(g,-)$ is injective on $J_g$ and $\Cal L$-definable,
\[
\ldim \ h\left(\{g\} \times \bigcup_{u \in K_g} (J_g\setminus \bigcup_i B^i_{g,u})\right)<k.
\]
By Lemma \ref{coneunion2}
\[
\ldim \ h\left(\bigcup_{g\in S} \{g\}\times \left(\bigcup_{u \in K_g} (J_g\setminus \bigcup_i B^i_{g,u})\right)\right)<k.
\]
By sub-induction hypothesis, it is only left to show that the restriction of $f$ to each
\[
h\left(\bigcup_{g\in S}\{ g\} \times \bigcup_{u \in K_g} B^i_{g,u}\right)
\]
satisfies the conclusion of ($2)_n$. Let $i\in \{1,\dots,p\}$. Let $h': M^{s+t+k} \to M^n$ map $(g,u,a)$ to $h(g,a)$.
Then
\[
h\left(\bigcup_{g\in S}\{ g\} \times \bigcup_{u \in K_g}  B^i_{g,u}\right)
=h'\left(\bigcup_{(g,u)\in K}\{ (g,u)\} \times  B^i_{g,u}\right).
\]
By ($3)_{n-1}$,  there is $q \in \N$ such that for every $j\in \{1,\dots,q\}$ there are an $A$-definable subset $K^j$ of $K$,  $k_{j} \in \{0,\dots,n\}$ and an $A$-definable uniform family of $k_j$-cones $\{Y_{g,u}^{j}\}_{(g,u)\in K^j}$ such that for each $(g,u) \in K$
\[
B^i_{g,u} = \bigcup \left\{Y_{g,u}^{j} :  (g,u) \in K^j\right\}.
\]
By Lemma \ref{largefibers}, we have that for each $j\in \{1,\dots,q\}$
\[
h'\left(\bigcup_{(g,u)\in K^j} \{(g,u)\}\times (Y^{j}_{g,u})\right)
\]
is an $A$-definable $k_{j}$-cone $h'(\cal Y^j)$, where $\cal Y^j$ denotes the inside family. Since
$$(f\circ h')(g, u, -)= (f\circ h)(g,-)=f^j(g,u,-)$$ on $Y_{g, u}^{j}$, we have that $f$ is fiber $\cal L_A$-definable with respect to $h'(\cal Y^{j})$.

\vskip.2cm \noindent   $\mathbf{(1)_{n} \,\Rarr\, (3)_{n}.}$ This is by a standard (but lengthy) compactness argument, which we  include for completeness. Let $\{X_t\}_{t\in M^m}$ be an $A$-definable family of subsets of $M^n$. Suppose that (3$)_n$ fails. Then for every finite collection $\{C_{t}^1\}_{t \in Y_1},\dots,\{C_t^p\}_{t \in Y_p}$ of $A$-definable uniform families of cones, there are $t\in M^m$ and $z\in M^n$ such that
\[
z \in X_t\setminus \left( \bigcup_{i=1}^p C_t^i \right).
\]
Since $\widetilde{\Cal M}$ is sufficiently saturated, there is $x\in M^m$ and $z\in X_x$ such that for every $A$-definable uniform family of cones $\{C_t\}_{t \in Y}$ either $x\notin Y$ or $z\notin C_x$. For the rest of the proof, we fix this $x$ and $z$.
By (1$)_n$ there is an $Ax$-definable $k$-cone $E\subseteq X_x$ with $z \in E$. This is not yet a contradiction, because we do not have a uniform family of cones such that $E$ is one element of this family. Let $k'= \dcl$-$\rank(z/AxP)$. By Lemma \ref{lem-rankcone}, there is an $Ax$-definable $k'$-cone $E'$ such that $z \in E'$. By (1$)_n$, there is an $Ax$-definable cone $F\subseteq E \cap E'$ such that $z\in F$. By Lemma \ref{lem-rankcone}, $F$ is a $k'$-cone. Therefore we can assume that $F=E$ and $k=k'$. It is left to show that there is an $A$-definable uniform family of $\{C_t\}_{t\in Y}$ such that
\begin{itemize}
\item[(I)] $C_t \subseteq X_t$ for each $t\in Y$,
\item[(II)] $x \in Y$ and $E=C_x$.
\end{itemize}
 Let $\cal J=\{J_g\}_{g\in S}$ be an $Ax$-definable uniform family of supercones in $M^k$, and $h: Z\subseteq M^{l+k} \to M^n$ an $\Cal L_{Ax}$-definable map, such that $E=h(\Cal J)$. Fix an $s\in S$ and $y \in J_s$ such that $h(s,y)=z$.\newline

\noindent Pick an $\Cal L_A$-definable function $h': Z' \subseteq M^{m+l+k} \to M^n$ such that $h'(x,-,-)=h$. Thus in particular, $Z'_x=Z$. Let $U \subseteq M^{m+l+k}$ be an $\Cal L_A$-definable cell such that $h'$ is continuous on $U$ and $(x,s,y) \in U$. Since $\dcl$-$\rank(z/AxP)=k$ 
we have that $\dim U_{x,s} = k$. By Lemma \ref{open-supercone}, $J_s \cap U_{x,s}$ is a supercone with closure $cl(U_{x,s})=cl(U_x)_{s}$. We now take
\begin{itemize}
\item an $A$-definable family  $\{S_t\}_{t\in M^m}$  of small subsets of $M^l$,
\item an $A$-definable family of $\{J'_{t,g}\}_{t\in M^m, g \in S_t}$ of subsets of $M^k$,
\end{itemize}
such that $S_x=S$ and $J'_{x,g} = J_g$ for all $g\in S$. Note that we make no further claims about the objects just defined, in particular we do not claim that they directly give rise to a family of cones satisfying (I) and (II). Let
\[
S'_{t} := \{ g \in S_t\ : \ J'_{t,g} \cap U_{t,g} \hbox{ is a supercone with closure } cl(U_{t,g}) \}.
\]
By Remark \ref{rem-supercone}(2), $(S'_{t})_{t\in Y}$ is an $A$-definable family. Let $Y' \subseteq M^m$ be the set of all $t\in Y$ such that $S'_t \neq \emptyset$ and
\[
h'\left(t,\bigcup_{g\in S'_t} \{ g\} \times (J'_{t,g} \cap U_{t,g})\right) \subseteq X_t.
\]
This set is $A$-definable. 
It is not hard to check that $s \in S'_x$ and hence $x \in Y'$. Denote
$$\cal K=\{J'_{t,g}\cap U_{t,g}\}_{t \in Y',g \in S'_t}.$$
By Lemma \ref{uniform family},
 $\cal K$ is an $A$-definable uniform family of supercones and
\[
\left\{h'\left(t,\bigcup_{g \in S'_t} \{ g \} \times (J'_{t,g}\cap U_{t,g})\right)\right\}_{t \in Y'}
\]
is an $A$-definable uniform family of $k$-cones satisfying (I) and (II).
\end{proof}

\begin{remark} $ $
\begin{enumerate}
  \item The proof of the Structure Theorem uses our standing assumption that $\widetilde{\Cal M}$ is sufficiently saturated. However, by Remark \ref{rem-supercone}(2), the Structure Theorem holds for any $\widetilde{\cal M}\models \widetilde T$.
  \item Using a standard compactness argument, the reader can verify that the following uniform version of (2) easily follows (from (2)): let $\{X_t\}_{t\in M^m}$ be an $A$-definable family of subsets of $M^n$ and $\{f_t:X_t\to M\}_{t\in M^m}$ an $A$-definable family of maps. Then the conclusion of (3) holds with every $f_t$ being fiber $\cal L_{At}$-definable with respect to $C_t^i$.
  \item We do not know whether we can have disjointness of the cones in the Structure Theorem. However, under one additional assumption, we do obtain it; see Theorem \ref{strongstrthm} below.
\end{enumerate}
\end{remark}

\subsection{Corollaries of the Structure Theorem}\label{sec-invariance}
We collect a few important corollaries of the Structure Theorem. The main result we are aiming for is Theorem \ref{stII}, a generalization of Corollary \ref{unary}.  We start with showing the invariance of the large dimension under definable bijections. Recall from Section \ref{sec-large dimension} that that the large dimension of a definable set $X\sub M^n$ is  the maximum $k\in \bb N$ such that there is a weak embedding of a supercone $J\sub M^k$ into $X$.

\begin{cor}[Invariance of large dimension]\label{bijection}
Let $f:X\to M^n$ be a definable injective function. Then $\ldim(X)=\ldim f(X)$.
\end{cor}
\begin{proof} Assume that $k\le \ldim(X)$. It suffices to show $k\le \ldim f(X)$. By the Structure Theorem, $X$ is the union of finitely many cones such that $f$ is fiber $\cal L$-definable with respect to each of them. By Corollary \ref{union}, one of them, say $h(\cal J)$ must be a $k$-cone. Pick any $g\in \pi(\cal J)$. Then $(f\circ h)(g, -):J\to M^n$ agrees with an $\cal L$-definable map on $J$ and it is injective. Therefore, 
$k\le \ldim f(X)$.
\end{proof}

The following is an easy  consequence of Structure Theorem (3).

\begin{cor}\label{uniform str2}
Let $D\sub M^m\times M^n$ an $A$-definable set. Then $D$ is a finite union of $A$-definable sets of the form
$$\bigcup_{t\in \Gamma} \{t\}\times C_t,$$
where $\Gamma\sub M^m$ is an $A$-definable cone and there is $k$ such that $\{C_t\}_{t \in \Gamma}$ is an $A$-definable uniform family of $k$-cones in $M^n$.
\end{cor}
\begin{proof} Left to the reader.
\end{proof}

We now establish certain desirable properties of large dimension. 


\begin{cor}\label{fibers}
 Let $X\sub M^{m+n}$ be an $A$-definable set and let $\pi_m(X)$ be its projection onto the first $m$ coordinates. Then
\begin{enumerate}
  \item For every $k\in \N$, the set of all $t\in \pi_m(X)$ such that $\ldim(X_t)=k$ is $A$-definable.
  \item Assume that for every $t\in \pi_m(X)$, $\ldim(X_t)=k$. Then $$\ldim(X)=\ldim(\pi_m(X))+k.$$
\end{enumerate}
\end{cor}
\begin{proof}
We observe that by \cite[Proposition 1.4]{vdd-dim}, we only need to prove both statements for $n=1$. Statement (1) is then immediate by Lemma \ref{lem:beg1} and Remark \ref{definability}(a).

(2). For $k=0$, this is by Lemma \ref{ldimproj}. For $k=1$, assume that $\ldim(\pi_m(X))=l$. By Structure Theorem (1), $\pi_m(X)$ is the finite union of cones $J_1, \dots, J_p$.  Assume that $J_i$  is a $k_i$-cone. By Lemma \ref{cone+1}, $T_i=J_i\times M$ is a finite union of $k_i+1$-cones, and by Corollary \ref{conedim}, each of them has large dimension $k_i+1$. Since $X$ is contained in $T_1\cup \dots \cup T_p$, it follows from Corollary \ref{union} that $\ldim(X)\le \max_i k_i+1=l+1.$

On the other hand, let $C$ be an $l$-cone contained in $\pi_m(X)$. By Remark \ref{definability}(b), there are definable $h_1, h_2:M^m\to M$ such that for every $t\in \pi_m(X)$, $X_t$ is co-small in $(h_1(t), h_2(t))$. By  Structure Theorem (2), $\pi_m(X)$ contains an $l$-cone $C'$ on which $h_1, h_2:M^m\to M$ are both fiber $\cal L$-definable. By Lemma \ref{cone+1}, it follows that $X$ contains an $l+1$-cone.
\end{proof}

\begin{lemma}\label{closure of cones}
Let $J_1, J_2\sub M^k$ be two supercones and $h_1:Z_1 \to M^n$, $h_2: Z_2\to M^n$  two $\cal L$-definable continuous injective maps, where $Z_i=sh(J_i)$ is the shell of $J_i$, $i=1, 2$.  
Then
\[
\dim\Big(h_1(Z_1)\cap h_2(Z_2)\Big)=k\,\,\,\implies\,\,\, \ldim\Big(h_1(J_1)\cap h_2(J_2)\Big)=k.
\]
\end{lemma}
\begin{proof}
Let
$$K_1=h_1^{-1} \big(h_1(Z_1)\cap h_2(Z_2)\big).$$ Then $K_1\sub Z_1$ and $\dim(K_1)=k$. By Lemma \ref{open-supercone}, $K_1\cap J_1$ contains a supercone $J$. Since $J\sub K_1$, we have
$$h_2^{-1}  h_1  (J)\sub Z_2.$$
Observe that $h_2^{-1}  h_1  (J)$ has large dimension $k$ and it is contained in the union of  $Z_2\sm J_2$ and $J_2$.  By Corollary \ref{clJ-J}, $Z_2\sm J_2$ has large dimension $<k$. Hence
\[
\ldim\big(h_2^{-1}  h_1  (J)\cap J_2\big)=k.
\]
Then  $\ldim \left( h_2(h_2^{-1}  h_1  (J)\cap J_2)\right)=k$. We observe
\[
h_2(h_2^{-1}  h_1  (J)\cap J_2)\sub h_1(J)\cap h_2(J_2) \sub h_1(J_1)\cap h_2(J_2),
\]
 proving that $h_1(J_1)\cap h_2(J_2)$ has large dimension $k$.
\end{proof}

We are now ready to prove the main result of this section. Statement (2) below is a higher dimensional analogue of Corollary \ref{unary}. To our knowledge, it has not been known even in the special case of dense pairs of o-minimal structures.

\begin{thm}\label{stII}$ $
\begin{itemize}
\item[(1)] Let $X\subseteq M^n$ be $A$-definable. Then there are disjoint $AP$-definable supercones $J_1,\dots,J_p \sub X$ such that
\[
 \ldim \left(X \sm \bigcup_{i=1}^p J_i \right) < n.
\]
\item[(2)] Every $A$-definable map $f: M^n\to M$ is given by an $\Cal L_{AP}$-definable map $F:M^n\to M$ off an $AP$-definable set of large dimension $ < n$.
\end{itemize}
Moreover, if $A\sm P$ is $\dcl$-independent over $P$, then in both statements the parameters from $P$ can be omitted.
\end{thm}
\begin{proof}
We again denote the above two statements by (1$)_n$ and (2$)_n$, and proceed by simultaneous induction on $n$. For $n=0$, they are both trivial. Suppose now that $n>0$ and (1)$_l$ and (2)$_l$ hold for every $l< n$. It is left to show  (1)$_{n}$ and (2)$_{n}$.

\vskip.2cm \noindent $\mathbf{(1)_{n}}$: Let $X\subseteq M^n$ and  $\pi : M^n \to M^{n-1}$ be the usual projection onto the first $n-1$ coordinates. By Corollary \ref{fibers}, the set
\[
\{ t \in X \ : \ \ldim(X_{\pi(t)}) = 0\}
\]
is $A$-definable and has $\ldim < n$. Therefore, we can reduce to the case that $\dim X_a = 1$ for all $a \in \pi(X)$. By Remark \ref{definability}(b), we may further assume that there are $A$-definable functions $h_1, h_2:M^{n-1}\to M\cup\{\pm\infty\}$ such that for every $a\in \pi(X)$, $X_a$ is co-small and contained in $(h_1(a), h_2(a))$. By (2$)_{n-1}$ there are $\Cal L_{AP}$-definable functions $H_1,H_2: M^{n-1} \to M$ and an $AP$-definable set $Z\subseteq M^{n-1}$ such that $\ldim(Z)<n-1$ and $H_1=h_1$ and $H_2=h_2$ on $M^{n-1}\sm Z$. By (1$)_{n-1}$ there are disjoint $AP$-definable supercones $J_1,\dots,J_p$ of $M^{n-1}$ such that $J_i \subseteq \pi(X)\sm Z$,
\begin{equation}\label{eq:st2}\tag{$\ast$}
\ldim \left((\pi(X) \sm Z)\sm \bigcup_{i=1}^p J_i \right) < n-1.
\end{equation}
By Lemma \ref{open-supercone} and cell decomposition in o-minimal structures, we can assume that $h_1,h_2$ are continuous on each $sh(J_i)$. Then each $K_i := \bigcup_{t \in J_i} \{t\} \times X_t$ is an $AP$-definable supercone. It follows immediately from Corollary \ref{fibers} and \eqref{eq:st2} that $\ldim( X \setminus \bigcup_{i=1}^p K_i) < n$, and that $K_1,\dots, K_p$ are disjoint.

\vskip.2cm \noindent   $\mathbf{(1)_{n} \,\Rarr\, (2)_{n}}$: Let $f: M^n\to M$ be $A$-definable. By Corollaries \ref{cor:corp1}  and \ref{nonlow}, there are $m\in \N$ and
\begin{itemize}
\item an $A$-definable set $Z \subseteq M^n$ with $\ldim (Z) < n$,
\item $\Cal L_{A \cup P}$-definable functions $f_i : Z_i \to M$ for $i=1,\dots,m$,
\end{itemize}
such that for each $a\in M^n\sm Z$ there is $i\in \{1,\dots,m\}$ such that $a\in Z_i$ and  $f(a)=f_i(a)$. Set
\[
X_i := \{ a \in M^n \ : \ f(a)=f_i(a) \wedge f(a)\neq f_j(a) \hbox{ for } j<i\}.
\]
Note that $X_i \cap X_j=\emptyset$ for $i\neq j$ and $\ldim(M^n \setminus \bigcup_{i=1}^m X_i) < n$. By (1$)_n$, for each $i=1,\dots,m$,  there are $AP$-definable supercones $J_{ik}\sub X_i$, $k=1,\dots,p_i$, such that $\ldim(X_i \sm \bigcup_{k=1}^{p_i} J_{ik}) < n$.
Note that $J_{ik} \cap J_{jl}=\emptyset$ for $i,j \in \{1,\dots,m\}$ with $i\neq j$ and $k=1,\dots,p_i$, $l=1, \dots, p_j$. Denote $V_{ik}= sh(J_{ik})$. By Lemma \ref{closure of cones},  for such $i,j,k$ and $l$, $V_{ik} \cap V_{jl}$ has dimension $<n$, and hence, since $V_{ik}$ and $V_{jl}$ are open, empty. Thus define $F: M^n \to M$ to map $x \in V_{ik}$ to $f_i(a)$ and  $x \notin \bigcup_{i}\bigcup_k V_{ik}$ to $0$. Note that this function is well-defined and $\Cal L_{AP}$-definable, since all $f_i$ and $V_{ik}$ are. Moreover, $F$ agrees with $f$ outside a set of large dimension $<n$; namely $X\sm \bigcup_{i=1}^m \bigcup_{k=1}^{p_i} J_{ik}$.\smallskip

The `moreover' clause follows from the above proof and Remark \ref{Aind}.
\end{proof}

We expect that Theorem \ref{stII} will find many applications in the future, and illustrate one here in the case of dense pairs.  Namely, we answer the following question from Dolich-Miller-Steinhorn \cite[page 702]{dms2}:  in dense pairs, is the graph of every $\emptyset$-definable unary map  nowhere dense? This property is known to fail if we allow parameters, as the example in Introduction shows. In \cite{dms1} the above authors isolate this property and examine it in the context of structures with o-minimal open core. 

\begin{prop}\label{prop-ndg} Let $\widetilde {\cal M}=\la \cal M, P\ra$ be a dense pair. Then the graph of every $\emptyset$-definable map $f:X\sub M\to M$ is nowhere dense.
\end{prop}
\begin{proof}
By Theorem \ref{stII}, $f$ agrees off a $\emptyset$-definable small set $S\sub X$ with an $\cal L_\emptyset$-definable function $F$. Clearly, the graph of $f_{\res X\sm S}$ is nowhere dense. We  therefore only need to prove that the graph of $f_{\res S}$ is nowhere dense. By Lemma \ref{small0}, $S\sub P$. By \cite[Lemma 3.1]{vdd-dense}, $f(S)\sub P$. By \cite[Theorem 3(3)]{vdd-dense}, $f$ is piecewise given by $\cal L$-definable functions, and hence its graph is nowhere dense.
\end{proof}


\subsection{Optimality of the Structure Theorem}\label{sec-optimality}
In this section, we prove that our Structure Theorem is in a certain sense optimal.

\begin{defn}
A \emph{strong cone} is a cone $h(\cal J)$ which, in addition to the properties of Definition \ref{def-cone}, satisfies:
\begin{itemize}
\item $h: \cal J\to M^n$ is injective.
\end{itemize}
\end{defn}

By \emph{Strong Structure Theorem} we mean the Structure Theorem where cones are replaced everywhere by strong cones. Below we give a counterexample to the Strong Structure Theorem and in the next section we point out a `choice property' that implies it. We will need the following lemma.

\begin{lemma}\label{embedsupercone_cont}
Let $J\sub M^n$ be a supercone  and $S\sub M^m$ small. Assume that $f:Z\sub M^n\to M^{m}$ is an $\cal L$-definable continuous  map with $J\sub Z$ that satisfies $f(J)\sub S$. Then $f_{\res J}$ is constant.
\end{lemma}
\begin{proof}
We work by induction on $n$. For $n=0$,  the statement is trivial. Now let  $n>1$ and assume we know the statement for all $J\sub M^{k}$ with $k<n$. Let $J\sub M^n$ and $f:Z\to S$ be as in the statement with $f(J)\sub S$.
For every $t\in \pi_1(J)$, by inductive hypothesis applied to $f(t, -):Z_t\to M^m$, there is unique $c_t\in S$ so that $f(\{t\}\times J_t)=\{c_t\}$. Since $f$ is continuous, and by definition of a supercone, for every $t\in \pi_1(Z)$, there is also unique $c_t\in S$ so that $f(\{t\}\times Z_t)=\{c_t\}$. We let $h:\pi_1(Z)\to M^m$ be the map given by $t\mapsto c_t$. If $f$ is not constant on $J$, there must be an interval $I\sub \pi_1(Z)$ on which $h$ is injective. But $I\cap \pi_1(J)\sub M$ is a supercone by Lemma \ref{open-supercone}, and  $h(I\cap \pi_1(J))\sub S$,  a contradiction. Therefore, $f$ is constant on $J$.
\end{proof}

\subsection*{Counterexample to the Strong Structure Theorem} We consider two closely related o-minimal structures: $\cal M=\la \R, <, +,  1, x\mapsto \pi x_{\res [0, 1]}\ra$ and its expansion $\cal M'= \la\R, <, +, 1, x\mapsto \pi x\ra$. It is well-known that $\Cal M$ does not define unrestricted multiplication by $\pi$ and that the theory of $\cal M'$ is the theory of ordered $\Q(\pi)$-vector spaces. We denote the language of $\cal M$ by $\cal L$ and the language of $\cal M'$ by $\cal L'$.\\

We now set $P:=\dcl(\emptyset)$. We first observe that $P=\Q(\pi)=\dcl_{\cal L'}(\emptyset)$. Indeed, since $\pi$ is $\Cal L_{\emptyset}$-definable, it is easy to see that $\Q(\pi)\subseteq P$. Note that $\Q(\pi)$ is a $\Q(\pi)$-vector space and therefore a model of the theory of $\Cal M'$. Thus $\dcl_{\cal L'}(\emptyset)\subseteq \Q(\pi)$.\\

Since $P=\Q(\pi)=\dcl_{\cal L'}(\emptyset)$, $\widetilde{\cal M}=\la \cal M, P\ra$ is a dense pair of models of the theory of $\cal M$ and $\la\cal M', P\ra$ is a dense pair of models of the theory of $\cal M'$. We will now show that the Strong Structure Theorem fails in $\widetilde{\cal M}$.
Being able to work in the two different dense pairs will be crucial. In the following, whenever we say a set is definable without referring to a particular language, we mean definable in $\widetilde{\cal M}$.\\

For $t\in M$, we denote by $l_t$ the straight line of slope $\pi$ that passes through $(t, 0)$. Define
$$U=\bigcup_{g\in P} l_g.$$
We will prove that $U$ is definable but not a finite union of strong cones.
By an \emph{endpart of $l_t$}, we mean $l_t \cap ([a, \infty)\times \R)$, for some $a\in \R$.\\

\noindent\textbf{Claim 1.} \emph{$U$ is definable.}
\begin{proof}[Proof of Claim 1]
For every $a\in M$, let $C_a=M\times [a, a+1)$ and $E_a\sub C_a\times C_a$ given by:
$$(x,y)E_a(x',y')\,\,\Lrarr\,\, y'-y=\pi(x'-x) \,\text{ and }\, |x'-x|\le 1.$$
Thus, if $(x, y)\in l_t\cap C_a$, then $[(x,y)]_{E_a}$ is the segment of $l_t$ that lies in $C_a$.
Define $p_a: C_a\to M^2$ via
$$p_a(x, y)= \text{ the midpoint of } [(x,y)]_{E_a},$$
and let $$Y_a=p_a(C_a\cap P^2).$$
Clearly, for $t\in P$, we have $l_t\cap P^2=\{(g, \pi (g-t)): g\in P\}$, and for $t\not\in P$, we have $l_t\cap P^2=\emptyset$. We claim that
$$U=\bigcup_{a\in M} Y_a,$$
and hence $U$ is definable.\smallskip

\noindent $(\subseteq)$. Let $(x, y)\in l_t$, $t\in P$. We claim that $(x, y)\in p_a(C_a\cap P^2)$, for $a=y-\frac{1}{2}$. Indeed, $(x, y)$ is the midpoint of $[(x, y)]_{E_a}=l_t\cap C_a$, and hence all we need is to find a point $(g_1, g_2)\in l_t\cap C_a\cap P^2$. Take any $g_2\in [a, a+1)\cap P$ and let $g_1=t+\frac{g_2}{\pi}\in P$. Then clearly $(g_1, g_2)\in l_t\cap C_a\cap P^2$ and hence $p_a(g_1, g_2)=(x,y)$.\smallskip

\noindent $(\supseteq)$. Let $(x, y)=p_a(g_1, g_2)\in p_a(C_a\cap P^2)$. Then $y-g_2=\pi(x-g_1)$. Hence, for $t=g_1-\frac{g_2}{\pi}$, we have $(x, y)\in l_t$.
\end{proof}

\noindent\textbf{Claim 2.} \emph{$U$ is not a finite union of strong cones.}
\begin{proof}[Proof of Claim 2]

First we observe that $\ldim(U)=1$. Indeed, $U$ contains infinite $\cal L$-definable sets, so $\ldim(U) \ge 1$. It cannot be $\ldim(U)=2$, by Lemma \ref{ldimproj} and since each vertical fiber is small (it contains at most one element of each $l_t$, $t\in P$). Therefore $\ldim(U)=1$.

Now assume, towards a contradiction, that $U$ is a finite union of strong cones. Let $h(\cal J)$ be one of them, where $\cal J=\bigcup_{g\in S} \{g\}\times J_g$, and $h:Z\to M^2$. In particular, $h$ is  injective on  $\cal J$. In the next two subclaims we make use of the expansion $\cal M'$ of $\cal M$ and the dense pair $\la\cal M', P\ra$.
\\

\noindent\textbf{Subclaim 1.} \emph{For every $g\in S$, $h(g,Z_g)$ must be contained in a unique $l_t$.}
\begin{proof}[Proof of Subclaim] Each of $l_t$ and the family $\{l_t\}_{t\in M}$ is now $\cal L'$-definable. Consider the $\cal L'$-definable and continuous map $f:Z_g\to M$ where
$$f(x)= t \,\,\Lrarr\,\, h(g, x)\in l_t.$$
By Lemma \ref{embedsupercone_cont} applied to $J=J_g$, $S=P$ and $f$, it follows that $h(g, J_g)$ must be contained in a unique $l_t$. By continuity of $h$, so does $h(g, Z_g)$.
\end{proof}

\noindent\textbf{Subclaim 2.} \emph{For every $t\in P$, there are only finitely many $g\in S$ such that $h(g, Z_g)\sub l_t$.}
\begin{proof}[Proof of Subclaim]
Assume, towards a contradiction, that for some $t\in P$ there are infinitely many $g\in S$ with $h(g, Z_g)\sub l_t$. For each $g\in S$, denote by $a_g$ the infimum of the projection of $h(g, Z_g)$ onto the first coordinate. By injectivity of $h$, for every two $g_1, g_2\in S$, we have  $h(g_1, J_{g_1})\cap h(g_2, J_{g_2})= \emptyset$. By Lemma \ref{closure of cones}, $h(g_1, Z_{g_1})\cap h(g_2, Z_{g_2})$ is finite (in fact, a singleton). Therefore, the set
$$\{a_g : g\in S \text{ and } h(g,Z_g)\sub l_t\}$$
is an infinite discrete $\cal L'(P)$-definable subset of $\R$, a contradiction.
\end{proof}
Since the subclaims hold for each of the finitely many strong cones, it turns out that for one of them, say $h(\cal J)$,  there is some $g\in \pi(\cal J)$ such that $h(g, Z_g)$ contains an endpart of $l_0$. So some endpart of $l_0$ is definable in $\widetilde{\cal M}$. But then its closure, which equals that endpart, is $\cal L$-definable. It follows easily that the full multiplication $x\mapsto \pi x$ is $\cal L$-definable, a contradiction.
\end{proof}

\subsection{Future directions}\label{sec-future} We now point out a key `choice property' which guarantees the Strong Structure Theorem. Indeed, together with Corollary \ref{uniform small} it implies a strengthened version of Lemma \ref{philippcor} below, which is enough.\\

 \noindent\textbf{Choice Property:} Let $h:Z\sub M^{n+k}\to M^l$ be an $\cal L_A$-definable continuous map and $S\sub M^n$   $A$-definable and small. Then there are $p, m \in \N$,  $\Cal L_A$-definable continuous maps $h_i :Z_i\subseteq M^{m+k}\to M^l$, $Y_i\sub M^m$ $A$-definable and small, and $A$-definable families $X_i\sub M^{m+k}$ with $X_{ia}\sub Y_i$, $i=1, \dots, p$, such that for every $a\in \pi(Z)$,
\begin{enumerate}
\item $h_i(-, a): X_{ia}\to M^l$ is injective, and

\item $h(S \cap Z_a, a)=\bigcup_i h_i(X_{ia}, a)$,
\end{enumerate}
where $\pi(Z)$ denotes the projection of $Z$ onto the last $k$ coordinates.





\begin{lemma}\label{philippcor2}
If $\widetilde M$ satisfies the Choice Property, then Lemma \ref{philippcor} holds with the additional conclusion that each $h_i: Z_i \to M^{n+1}$ is injective.
\end{lemma}
\begin{proof} We first claim that there are $m, p\in \N$, and for each $i=1,\dots,p$,
 an $\cal L_A$-definable continuous function $h_i:Z_i \subseteq M^{m+n}\to M$, an $A$-definable small set $S_i \sub M^m$ and an $A$-definable family $Y_i\sub S_i\times C$,  such that for all $a\in I$,
\begin{enumerate}
\item $h_i(-, a): Y_{ia}\to M$ is injective,

\item $X_a=\bigcup_i h_i(Y_{ia}, a)$,

\item $\{h_i(Y_{ia}, a)\}_{i=1, \dots, p}$ are  disjoint.
\end{enumerate}
Indeed, apply the Choice Property to each $h_i$ from Corollary \ref{uniform small} to get (1) and (2). 
For (3), recursively replace $Y_{ia}$, $1<i\le l$, with the set consisting of all $z\in Y_{ia}$ such that $h_i(z, a)\not\in h_j(Y_{ja}, a)$, $0<j<i$.
We now have:
$$X=\bigcup_{a\in C} \{a\} \times X_a =\bigcup_i \bigcup_{a\in C}  \{a\} \times h_i (Y_{ia}, a).$$
From this point on the argument continues identically with the corresponding part of Lemma \ref{philippcor}, noting in the end that, by (1), each $\hat{h}_i:\bigcup_{g\in S_i} \{g\}\times V_{ij}\to M^{n+1} $ turns out to be injective.
\end{proof}

\begin{theorem}\label{strongstrthm}
If $\widetilde M$ satisfies the Choice Property, then  the Structure Theorem holds with cones replaced  by strong cones. Moreover, the unions of cones in Structure Theorem are disjoint.
\end{theorem}
\begin{proof} The reader can check that Lemmas \ref{cone+1} and \ref{largefibers} hold with cones replaced everywhere by strong cones, with identical proofs. Moreover, the Choice Property for $k=0$ implies that every $0$-cone is a finite union of strong $0$-cones, and hence it is easy to obtain Lemma \ref{lem:2n0} with strong $0$-cones in place of $0$-cones, as well. It is then a (rather lengthy) routine to check that the proof of the current statement is, again, identical with that of the Structure Theorem, with cones replaced everywhere by strong cones and with the further condition that the unions of cones can be taken to be disjoint. In the proof, Lemma \ref{philippcor} has to be replaced by Lemma \ref{philippcor2} in order to get strong cones and not just cones. The injectivity of the $h_i$'s in Lemma \ref{philippcor2} guarantees the disjointness of the cones. We leave the details to the reader.
\end{proof}



The counterexample to the Strong Structure Theorem relies on a somewhat unnatural condition on \cal M. In \cite{egh2}, we establish the Choice Property for a collection of structures $\widetilde{\cal M}=\la \cal M, P\ra$, such as when $\cal M$ is a real closed field, or when $P$ is a dense independent set. More generally, we can ask the following question.

\begin{question}
Under what assumptions on $\cal M$ or $\widetilde{\cal M}$ does the Choice Property hold? 
\end{question}

There are other ways in which one could try to improve the Structure Theorem. In general, a supercone $J\sub M^n$ does not contain a product of supercones in $M$. For example, let $\widetilde{\cal M}=\la \cal M, P\ra$ be a dense pair of real closed fields and  $J\sub M^2$ with
$$J=\bigcup_{a\in M} \{a\}\times (M\sm aP).$$
It is natural to ask whether $J$ contains an image of such product under $\cal L$-definable map. More generally, one could ask the following question.

\begin{question}\label{productcones}
Would the Structure Theorem remain true if we defined:
\begin{enumerate}
  \item supercones in $M^k$ to be products $J_1\times \dots \times J_k$, where each $J_i$ is a supercone in $M$?

  \item $k$-cones to be of the form $h(S\times J)$? (That is,   $h$ and $S$ are as before, but $J_g=J$ in Definition \ref{def-cone} is fixed.)
\end{enumerate}
\end{question}

In subsequent work \cite{el-opt}, we refute both questions, showing that our definitions and Structure Theorem are optimal in yet another way.

\section{Large dimension versus $\scl$-dimension}\label{sec-equaldim}

In this section we use our Structure Theorem to establish the equality of the large dimension with the `$\scl$-dimension' arising from a relevant pregeometry in \cite{beg}. In Section \ref{sec-groups} we use this equality to set forth the analysis of groups definable  in $\widetilde{\cal M}$.

We start by quoting \cite[Definition 28]{beg}, which was given independently from, and in complete analogy with, \cite[Definition 5.2]{el-sbdgroups}. 

\begin{definition} The  \emph{small closure} operator $\scl: \cal P(M)\rightarrow \cal P(M)$ is defined by:
\[
a\in \scl(A) \Lrarr \text{ $a$ belongs to an $A$-definable small set.}
\]
\end{definition}

In \cite{beg} $\scl$ was shown to define a pregeometry under certain assumptions (in addition to their basic tameness conditions). We show that in the current context $\scl$ always defines a pregeometry. This follows from the first equality below, which is proved using only results from Section \ref{sec-small}. In the interests of completeness, we also prove a second equality, using the Structure Theorem. Recall that $\dcl(A)$ denotes the usual definable closure of $A$ in the o-minimal structure $\cal M$.

\begin{lemma}\label{scl-dcl}
$\scl(A)= \dcl (P\cup A)= \dcl_{\cal L(P)}(P\cup A)$.
\end{lemma}
\begin{proof}
$\scl ( A)\sub \dcl(P\cup A).$ Let $b\in \scl(A)$. Then there are an $\cal L(P)$-formula $\varphi(x, y)$ and $a\in A^l$, such that $\varphi(\cal M, a)$ is small and contains $b$. Consider the $\emptyset$-definable family $\{\varphi(\cal M, t)\}_{t\in M^l}$. By Remark \ref{definability}(a), the set $I$ consisting of all $t\in M^l$ such that $\varphi(\cal M, t)$ is small is $\emptyset$-definable. Of course, $I$ contains $a$. By Corollary \ref{uniform small}, there is an $\cal L_\emptyset$-definable function $h:M^{m+l}\to M$ such that for all $t\in I$, $\varphi(\cal M, t)\sub h(P^m, t)$. Therefore $b\in h(P^m, a)$, and $b\in \dcl  (P\cup A)$.\smallskip

$\scl ( A)\supseteq \dcl(P\cup A).$ Let $b\in \dcl (P\cup A)$. Then there is an $\cal L_\emptyset$-definable $h:M^{m+l}\to M$ and $a\sub A^l$ such that $b\in h(P^l, a)$. But the latter set is small, hence $b\in \scl(A)$.\smallskip



$\dcl (P\cup A)= \dcl_{\cal L(P)}(P\cup A)$. It suffices to show $\dcl_{\cal L(P)}(P\cup A)\sub \dcl (P\cup A)$. Let $b=f(a)$, where $a\sub P\cup A$ and $f$ is $\emptyset$-definable. By Structure Theorem, there is a $\emptyset$-definable cone $h(\cal J)$, where $h$ is $\cal L_\emptyset$-definable, containing $a$ on which $f$ is fiber $\cal L_\emptyset$-definable. Denote $S=\pi(\cal J)$. Let $g\in S$ and $t\in J_g$ be so that $a=h(g, t)$. Since $h(g, -):M^k\to M^n$ is $\cal L_{g}$-definable and injective, $t\in \dcl(P\cup A\cup S)$. Moreover, $S$ is $P$-bound over $\emptyset$ (Lemma \ref{lem:beg2}) and hence $t\in \dcl(A\cup P)$. Since $f h(g, -)$ agrees with an $\cal L_{A\cup P}$-definable map on $J_g$, it follows that
$$b=f(h(g, t))\in \dcl (A\cup P).$$
\end{proof}

\begin{remark} In general $\dcl(P\cup A)\neq \dcl_{\cal L(P)}(A)$. For example, let $\la \Cal M,\Cal N \ra$ be a dense pair of real closed fields and let $\Cal N_0$ be a real closed subfield of $\Cal N$. Then $\dcl_{\cal L(P)}(\Cal N_0) = \Cal N_0$ by \cite[Lemma 3.2]{vdd-dense}.
\end{remark}


The following corollary is then immediate.

\begin{cor}\label{scl-preg}
The small closure operator $\scl$ defines a pregeometry.
\end{cor}

\begin{definition}
Let $A,B\subseteq M$. We say that $B$ is \emph{$\scl$-independent over $A$} if for all $b\in B$, $b\not\in \scl\big(A\cup (B\setminus \{b\})\big)$.
A maximal $\scl$-independent  subset of $B$ over $A$ is called \emph{a basis for $B$ over $A$}.
\end{definition}

By the Exchange property for $\scl$, any two bases for $B$ over $A$ have the same cardinality. This allows us to define the
\emph{rank of $B$ over $A$}:
\[\rank(B/A)=\text{ the cardinality of any basis of $B$ over $A$}.\]




In complete analogy with the corresponding fact for $acl$ in a pregeometric theory, we can prove:

\begin{lemma}
If $p$ is a partial type over $A\sub M$ and $a\models p$ with $\rank(a/A)=m$, then for any set $B\supseteq A$ there is $a^\prime \models p$ such that $\rank(a^\prime/B)\ge m$.
\end{lemma}
\begin{proof}
The proof of the analogous result for the rank coming from $acl$ in a pregeometric theory is given, for example, in \cite[page 315]{g}. The proof of the present lemma is word-by-word the
same with that one, after replacing an `algebraic formula' by a `formula defining a small set' in the definition of $\Phi_B^m$ (\cite[Definition 2.2]{g}) and
the notion of `algebraic independence' by that of `$\scl$-independence' we have here.
\end{proof}


It follows that the corresponding dimension of partial types and definable sets is well-defined and independent of the choice of the parameter set.

\begin{definition}
Let $p$ be a partial type over $A\subset M$. The \emph{$\scl$-dimension of $p$} is defined as follows:
\[
\cldim(p)= \max\{\rank(\bar{a}/A): \bar{a}\subset M \text{ and } \bar{a}\models p\}.
\]
Let $X$ be a definable set. Then the \emph{$\scl$-dimension of $X$}, denoted by $\cldim(X)$ is the dimension of its defining formula.
\end{definition}

We next prove the equivalence of the $\scl$-dimension and large dimension of a definable set. First, by a standard routine, using the saturation of $\widetilde{\cal M}$, we observe the following fact about supercones.

\begin{fact}\label{rankJ}
Let $J\sub M^k$ be an $A$-definable supercone. Then $J$ contains a tuple of rank $k$ over $A$.
\end{fact}

\begin{prop}\label{equaldim}
For every definable $X\sub M^n$.
$$\ldim(X)=\cldim(X).
$$
\end{prop}
\begin{proof} We may assume that $X$ is $\emptyset$-definable.

$\le$.  Let $f:M^k\to M^n$ be an $\cal L$-definable injective function and  $J\sub M^k$ a supercone, such that $f(J)\sub X$. Suppose both $f$ and $J$ are defined over $A$. We need to show that $f(J)$ contains a tuple $b$ with $\rank(b/\emptyset)\ge k$. By Fact \ref{rankJ}, $J$ contains a tuple  $a$ of rank $k$ over $A$. Let $b=f(a)$. Since $f$ is injective, we have $a\in \dcl (Ab)$ and $b\in \dcl (Aa)$. In particular, $a\in scl(Ab)$ and $b\in \scl(Aa)$. So $a$ and $b$ have the same rank over $A$. Hence,
$$\rank(b/\emptyset)\ge \rank(b/A)=\rank(a/A)=k.$$

$\ge$. Let $b\in X$ be a tuple of rank $k$. By the Structure Theorem, $b$ is contained in some $l$-cone $C\sub X$. We prove that $l\ge k$. Let $C=h(\cal J)$, where $\cal J$ is a uniform family of supercones in $M^l$. Suppose $b=h(g, a)$, for some $g\in \pi(\cal J)$ and $a\in J_g$. Since $h(g, -)$ is $\cal L_g$-definable and injective, we have $a\in \dcl (gb)$ and $b\in \dcl (ga)$. In particular, $a\in \scl(gb)$ and $b\in \scl(ga)$. Hence $a$ and $b$ have the same rank over $g$. But $a\in J\sub M^l$ and, hence,
$$k=\rank(b/g)=\rank(a/g)\le l.$$
\end{proof}

We next record several properties of the rank and large dimension, for future reference. By $\dcl$-$\rank$ we denote the usual rank associated to $\dcl$.

\begin{lemma}\label{prop-rank} For every $a\in M$ and $A\sub M$, we have
\begin{enumerate}
  \item $\scl(A\cup P)=\scl(A)$
  \item $\rank(a/AP)=\rank(a/A)=\text{$\dcl$-$\rank(a/AP)$}.$
\end{enumerate}

\end{lemma}
\begin{proof}
Immediate from Lemma \ref{scl-dcl} and the definitions.
\end{proof}

\begin{lemma}  \label{ldimlem} Let $X, Y, X_1, \dots, X_k$ be definable sets. Then:
\begin{enumerate}
\item $\ldim(X)\le \dim(cl(X))$. Hence, if $X$ is $\cal L$-definable, $\ldim X= \dim X$.

\item $X\sub Y\sub M^n\,\Rarr \, \ldim(X)\le\ldim(Y)\le n$.

\item $X$ is small if and only if $\ldim(X)=0$.

\item If $C$ is a $k$-cone, then $\ldim(C)=k$.

\item $\ldim(X_1\cup\dots\cup X_l)=\max\{\ldim(X_1), \dots, \ldim(X_l)\}$.

\item $\ldim(X\times Y)=\ldim(X)+\ldim(Y)$.
\end{enumerate}
\end{lemma}

\begin{proof}
(1). Assume $X$ is $A$-definable and let $a\in X$ with $\rank(a/A)=\ldim(X)$. Since $a\in cl(X)$, we have
$$\ldim(X)=\rank(a/A)=\text{$\dcl$-$\rank(a/A\cup P)$}\le \text{$\dcl$-$\rank(a/A)$}\le \dim cl(X).$$


Now, if $X$ is $\cal L$-definable, $\ldim(X)\le \dim cl(X)=\dim X$. On the other hand, if $\dim X=k$, one can $\cal L$-definably embed a $k$-box in $X$ which of course is a $k$-cone.

 \noindent (2)-(5) were proved in Section \ref{sec-cones}, and (6) is by virtue of $\scl$ defining a pregeometry.
\end{proof}


\subsection{$\scl$-generics}


For a treatment of the classical notion of $\dcl$-generic elements, see, for example, \cite{pi-groups}. Here we introduce the corresponding notion for $\scl$.

\begin{defn}
Let $X\sub M^n$ be an $A\cup P$-definable set, and let $a\in X$. We say that $a$ is a \emph{$\scl$-generic element of $X$ over $A$} if it
does not belong to any $A$-definable set of large dimension $<\ldim(X)$. If $A=\emptyset$, we call $a$ a \emph{$\scl$-generic element of
$X$}.
\end{defn}

By saturation,  \emph{$\scl$-generic elements always exist}. More precisely, every $A\cup P$-definable set $X$ contains an $\scl$-generic element over $A$. Indeed, by Compactness and Lemma \ref{ldimlem}(5), the collection of all formulas which express that $x$ belongs to $X$ but not to any $A$-definable set of large dimension $<\ldim(X)$ is consistent.

Two $\scl$-generics are called \emph{independent} if one (each) of them is $\scl$-generic over the other. The facts that $\scl$ defines a pregeometry and that the $\cldim$ agrees with $\ldim$ imply:

\begin{fact}\label{generic} Let $G=\la G, *\ra$ be a $\emptyset$-definable group. If $a, b\in G$ are independent $\scl$-generics, then so are $a$ and $a* b^{-1}$.
\end{fact}
\begin{proof} We have
$$\rank(b/a)=\rank(a* b^{-1}/ a).$$
So if $b$ is $\scl$-generic over $a$, then so is $a*b^{-1}$.
\end{proof}

Note that none of the notions `$\dcl$-generic element' and `$\scl$-generic element' implies the other, but, by Lemma \ref{prop-rank}, if $X$ is $A\cup P$-definable  and $a\in X$, we have:
$$\text{$a$ is $\scl$-generic over $A\cup P$} \Lrarr \text{$a$ is $\scl$-generic over $A$}\Lrarr \text{$a$ is $\dcl$-generic over $A\cup P$}.$$


\section{Definable groups}\label{sec-groups}

In this section we obtain our  main application of the Structure Theorem. We fix a $\emptyset$-definable group  $G=\la G, *, 0_G\ra$  with $G \sub M^n$ and $\ldim(G)=k$ and prove a local theorem for $G$: around $\scl$-generic elements the group operation is given by an $\cal L$-definable map.\\

\noindent\textbf{A convention on terminology.} When we say that $h(J)$ is a $k$-cone, we mean that there is a $k$-cone $h'(\cal J)$ and $g\in \pi(\cal J)$, such that $J=J_g$ and $h(-)=h'(g, -)$. We call $h(J)$ $A\cup P$-definable, if $h'(\cal J)$ is $A$-definable. Likewise, when we say that $\cal T=\{\tau_t(J_t)\}_{t\in X}$ is a  uniform family of $k$-cones, we mean that there is a uniform family $\cal C=\{C_t\}_{t\in X}$ of $k$-cones as in  Definition \ref{def-unifcones} and $g\in \bigcap_{t} S_t$, such that for every $t\in X$, $J_t=Y_{t, g}$ and $\tau_t(-)=h(t, g, -)$. We call $\cal T$ $A\cup P$-definable if $\cal C$ is $A$-definable. We write $\cal T=\{\tau(J_t)\}_{t\in X}$, if for all $t, s$, we have $\tau_t=\tau_s$.

\begin{lemma}\label{conekl}
Let $\{C_t=\tau_t(J_t)\}_{t\in \Gamma}$ be a uniform family of $k$-cones in $M^n$ and $\Gamma\sub M^m$ a $k'$-cone. Then
$$C=\bigcup_{t\in \Gamma} \{t\}\times C_t$$
is a $k'+k$-cone.
\end{lemma}
\begin{proof} Assume $\Gamma=\tau(\cal I)$, where $\cal I=\bigcup_{s\in S}\{s\}\times I_s$ and $S\sub M^{p}$, and for every $t\in \Gamma$,
$$C_t=h(t, g, Y_{t, g}),$$ for some fixed $g\in \bigcap_t S_t$, and $h$, $\{Y_{t, g}\}_{t\in \Gamma}$ as in Definition \ref{def-unifcones}. We define
$$ h':Z\sub M^{p+k'+k}\to M^{m+n}: (s, x, y)\mapsto (\tau(s,x), h(\tau(s, x), g, y)),$$
for a suitable $Z$, and, for every $s\in S$,
$$J_{s}=\bigcup_{x\in I_{s}} \{x\}\times Y_{\tau(s, x), g}.$$
The reader can verify that
$$C= h'\left(\bigcup_{s\in S} \{s\}\times J_{s}\right)$$
is a $k'+k$-cone, as required.
\end{proof}

 \begin{lemma}\label{Ct} Let $h(J)$ be a $k$-cone, and $\{D_t\}_{t\in \Gamma}$ a definable family of sets, such that for each $t\in \Gamma$, $\ldim(D_t)=k$ and $D_t\sub h(J)$. Then there is a uniform definable family of $k$-cones $\{C_t=h(Y_t)\}_{t\in \Gamma}$ with $C_t\sub D_t$.
 \end{lemma}
 \begin{proof} This follows from a uniform version of Theorem \ref{stII}(1), which can be proved easily via a standard compactness argument. Indeed, for every $t\in\Gamma$, let $X_t=h^{-1}(D_t)\sub J$. So $\ldim(X_t)=k$. By the uniform Theorem \ref{stII}(1), we can find a uniform family of supercones $Y_t\sub X_t$. Then $C_t=\{h(Y_t)\}_{t\in\Gamma}$ is as required.
 \end{proof}



\begin{lemma}\label{projcone} Let $X\sub M^n$ be a $\emptyset$-definable set of large dimension $k$, $(a, b)$ an $\scl$-generic element of $X\times X$, and $D\sub X\times X$ a $\emptyset$-definable $2k$-cone containing $(a, b)$. Then there is a $P$-definable uniform family of $k$-cones $\{E_t=\tau_t( J_t)\}_{t\in T}$, where $T$ is a $k$-cone containing $a$, such that $b\in\bigcap_{t\in T} cl(E_t)$ and
$$(a, b)\in \bigcup_{t\in T} \{t\}\times E_t \sub D.$$
\end{lemma}
\begin{proof}
By Corollary \ref{uniform str2}, and since $(a,b)$ is  $\scl$-generic of $X\times X$, it is contained  in a $\emptyset$-definable set of the form
$$\bigcup_{t\in \Gamma} \{t\}\times C_t\sub D,$$
where $\Gamma\sub X$ is a cone and there is $l$ such that $\{C_t\}$ is an $\emptyset$-definable uniform family of $l$-cones contained in $X$. Write
$$C_t=h\left(\{t\} \times \left(\bigcup_{g\in S_t} \{g\}\times Y_{t,g}\right)\right),$$
as in Definition \ref{def-unifcones} where $h: Z \to M^{n}$.  Since $a\in \Gamma\subseteq X$ and $a$ is a $\scl$-generic element of $X$, $\Gamma$ must be a $k$-cone. Thus there is a supercone $J_0\subseteq M^k$ and an $\Cal L_P$-definable, continuous and injective map $f: U \subseteq M^k \to M^n$ such that $f(J_0)=\Gamma$. Let $\hat a \in M^k$ such that $f(\hat a)=a$.  Because $(a,b)$ is an $\scl$-generic element of $X\times X$, $\hat a$ is $\scl$-generic over $b$. Since $b\in C_a\sub X$ and $b$ is a $\scl$-generic element of $X$ over $a$, $C_a$ must be a $k$-cone, and hence $l=k$. Fix $g \in S_a$ such that $b \in h(a,g,Y_{a,g})$. Because $\hat a$ is $\scl$-generic over $b$, there is an open box $B\subseteq M^k$ containing $\hat a$ such that $b \in cl(h(f(x),g,Z_{f(x),g}))$ for every $x\in B$.
By density of $P$ we can assume that $B$ is $\cal L_{P}$-definable. By Lemma \ref{open-supercone}, $J_0 \cap B$ is a supercone. Hence
\[
(a,b) \in \bigcup_{t \in f(J_0\cap B)} \{t\} \times h(t,g,Y_{t,g})
\]
and $b\in \bigcap_{t \in f(J_0\cap B)} cl(h(t,g,Y_{t,g}))$. Set $E_t=h(t,g,Y_{t,g})$.
\end{proof}


\begin{remark}
In general, there is no $\{E_t\}_{t\in T}$ as above so that $b\in\bigcap_{t\in T} E_t$. For example, let $\widetilde{\cal M}=\la \cal M, P\ra$ be a dense pair of real closed fields, $X= M$
$$D=\bigcup_{c\in M} \{c\}\times (M\sm cP),$$
and $(a,b)$ any element of $D$.
\end{remark}

\begin{cor}\label{loclinear} Let $X\sub M^n$ be a $\emptyset$-definable set of large dimension $k$. Let $(a,b)$ be an $\scl$-generic element of $X\times X$ and $f:X\times X\rarr X$ a $\emptyset$-definable function. Then there is a $P$-definable uniform family of $k$-cones $\{E_t=\tau_t(J_t)\}_{t\in T}$, where $T$ is a $k$-cone containing $a$, such that $b\in \bigcap_{t\in T} cl(E_t)$ and $f$ agrees with an $\cal L_P$-definable continuous map on
$$E=\bigcup_{t\in T} \{t\}\times E_t.$$
\end{cor}
\begin{proof}
By  the Structure Theorem, there is a $\emptyset$-definable $2k$-cone $D\sub G\times G$ that contains $(a, b)$ and such that $f$ agrees with an $\cal L_P$-definable continuous map on $D$. The statement then follows from Lemma \ref{projcone}.
\end{proof}

We are now ready to prove the local theorem for definable groups.

\begin{thm}[Local theorem for definable groups]\label{arounda}
Let $a$ be an $\scl$-generic element of $G$. Then there is a $2k$-cone $C\sub G\times G$, whose closure contains $(a, a)$, and an $\cal L$-definable continuous map $F:Z\sub M^n\times M^n\to M^n$, such that for every $(x, y)\in C$,
\[
x  * a^{-1} *  y=F(x, y).
\]
Moreover, $F$ is a homeomorphism in each coordinate.
\end{thm}



\begin{proof} 
Let $a_1\in G$ be $\scl$-generic over $a$, and let $a_2=a_1^{-1}  *  a$. By Fact \ref{generic}, $a, a_1, a_2$ are pairwise independent.
By the Structure Theorem, for $i=1, 2$, there is a $Pa_i$-definable $k$-cone $C_i=h_i(J_i)\sub G$ containing $a$,
and $\cal L_{P a_i}$-definable continuous $f_i:Z_i\sub M^n\to M^n$ such that for every $x\in C_1$,
$$x *a_2^{-1}=f_2(x)$$
and for every $y\in C_2$,
$$a_1^{-1} *y  =f_1(y).$$
Observe that $f_2(a)=a_1$ and $f_1(a)=a_2$.


We now look at the independent $\scl$-generic elements $a_1$ and $a_2$. By Corollary \ref{loclinear},  there is a $P$-definable uniform family of $k$-cones $\{E_t=\tau_t(J_t)\}_{t\in T}$ in $G$, where $T\sub G$ is a $k$-cone containing $a_1$ and $a_2\in\bigcap_{t\in T} cl(E_t)$, such that $*$ agrees with an $\cal L_P$-definable continuous map $f:Z\sub M^n\times M^n\to M^n$ on
$$E=\bigcup_{t\in T} \{t\}\times E_t.$$
 Observe that  $(a, a_i)$ is also $\scl$-generic of $G\times G$. Moreover, since $a_2$ is $\dcl$-generic of $G$ over $P$, there is an $\cal L_{P}$-definable $B$ of dimension $k$ with
$$a_2\in B\sub  \bigcap_{t\in T} cl(E_t).$$

\noindent\textbf{Claim.} \emph{For every $t\in T$, $f_1^{-1}(E_t)\cap h_1(J_1)$ has large dimension $k$.}

 \begin{proof}[Proof of Claim] Let $F_t= f_1^{-1}\tau_t$. Since  $a$ belongs to the $\cal L_{Pa_1}$-definable set $f_1^{-1}(B)\cap h_1(cl(J_1))$ and it is $\scl$-generic over $a_1$, the set
$$f_1^{-1}(B)\cap h_1(cl(J_1)) \sub f_1^{-1}(cl(\tau_t(J_t))\cap h_1(cl(J_1))$$
has dimension $k$. This implies that $F(cl(J_t))\cap h_1(cl(J_1))$ has dimension $k$. By Lemma \ref{closure of cones}, $f_1^{-1}(E_t)\cap h_1(J_1)=F(J_t)\cap h_1(J_1)$ has large dimension $k$.
 \end{proof}

Now, since $a$ belongs to the $Pa_2$-definable set $f_2^{-1}(T)\cap h_2(J_2)$ and it is $\scl$-generic over $a_2$, it must also belong to a $Pa_2$-definable $k$-cone
$$\Gamma\sub f_2^{-1}(T)\cap h_2(J_2).$$
For every $t\in \Gamma$, we let
$$D_t= f_1^{-1}(E_{f_2(t)})\cap h_1(J_1).$$
By  Claim, $\ldim(D_t)=k$.  Since every $D_t\sub h_1(J_1)$, by Lemma \ref{Ct}, we can find a uniform definable family of $k$-cones
$$C_t=h_1(Y_t)\sub D_t,\,\, t\in \Gamma,$$
where $Y_t\sub J_1$ is a supercone in $M^k$, and $a\in \bigcap_{t\in \Gamma} C_t$. By Lemma \ref{conekl}, the set
$$C=\bigcup_{t\in \Gamma} \{t\}\times C_t$$
is a $2k$-cone. 
We can now conclude as follows. For every $(x, y)\in C$,
$$x *a^{-1}* y= (x *a^{-1}* a_1)* (a_1^{-1}* y)= f_2(x) *f_1(y) =f(f_2(x), f_1(y)).$$
Set $$F(x, y)=f(f_2(x) , f_1(y)) : M^n\times M^n \to M^n.$$

For  the ``moreover" clause, we need to check that (a) each $f_i$ can be chosen to be a homeomorphism, and (b) $f$ can be chosen to be a homeomorphism in each coordinate.  The former fact follows from the $\scl$-genericity of $a$ over each $a_i$ and the injectivity of each $x\mapsto x* a_i^{-1}$, and the latter fact from the $\scl$-genericity of $(a_1, a_2)$ and the injectivity of $*$ in each coordinate.
\end{proof}

\begin{remark}
We observe that we cannot always have $C=C'\times C''$, where $C', C''$ are $k$-cones containing $a$. For example, consider the group $\CH = \la H=[0, 1),  +\, mod\, 1\ra$ in the real field, and  let $T=\Q^{rc}\cap H$. Now let $g:H\to M$ be the translation $x\mapsto 2+x$ on $T$, and identity elsewhere. Let $G$ be the induced group on $(H\sm T)\cup g(T)$. Clearly, $G$ is definable in $\widetilde M =\la\R, \Q^{rc}\ra$, and it is easy to verify that the above observation holds for every $a\in G$. Of course, the conclusion of Theorem \ref{arounda} holds for every $a\in H\sm T$, by letting $\Gamma=H\sm T$, $C_t=H\sm (T\cup (T-t)\cup (1+T-t))$ and $f= +\, mod\, 1$. Moreover, we can achieve $C=C'\times C'$, but only up to definable isomorphism. It is reasonable to ask whether that is always true, and we include some relevant (in fact, stronger) questions at the end of this section.
\end{remark}

We expect that the above local theorem will play a crucial role in forthcoming analysis of  groups definable in $\widetilde M$. The ultimate goal would be to understand definable groups in terms of $\cal L$-definable groups and small groups. Motivated by the successful analysis of semi-bounded groups in \cite{ep-sel2} and the more recent \cite{bm}, we conjecture the following statement. This is a reformulation of \cite[Conjecture 1]{el-bsl}.

\begin{conj}\label{conjecture}
Let $\la G, *\ra$ be a definable group.  Then there is a short exact sequence
$$
\begin{diagram}
\node{0}\arrow{e}\node{\mathcal B} \arrow{e}\node{\cal U} \arrow{s,r}{\tau}\arrow{e}\node{K} \arrow{e} \node{0}\\
\node[3]{G}
\end{diagram}
$$
where
\begin{itemize}
\item $\cal U$ is $\bigvee$-definable
  \item $\mathcal B$ is $\bigvee$-definable in $\mathcal L$ with $\dim(B) =\ldim (G)$.
  \item $K$ is definable and small
  \item $\tau:\cal U\to G$ is a surjective group homomorphism and
  \item all maps involved are $\bigvee$-definable.
\end{itemize}
\end{conj}

The conjecture is in a certain sense optimal: we next produce an example of a definable group $G$ which is \emph{not} a direct product of an $\cal L$-definable group by a small group. Using known examples of $\cal L$-definable groups $B$ from \cite{pes, str}, which are not direct products of one-dimensional subgroups, it would be easy to provide such an $G$ by restricting some of the one-dimensional subgroups of the universal cover of $B$ to the subgroup $P$ (say, in a dense pair). Our example below, however, is not constructed in this way, as it is \emph{not} a subgroup of the examples in \cite{pes, str}.

\begin{example} Let $\widetilde{\mathcal M}=\la \mathcal M, P\ra\models  T^d$. Let $G=\la P\times [0,1), \oplus, 0\ra$, where $x\oplus y= x+y \mod (1, 1)$; that is,
$$
x\oplus y=\begin{cases}
x+y, & \text{ if $x+y\in P\times [0, 1)$}\\
x+y-(1, 1), & \text{ otherwise}
\end{cases}
$$
Then $G$ is clearly not small. But it cannot contain any non-trivial $\cal L$-definable subgroup. Indeed, by o-minimality, every $\cal L$-definable subset of $P\times [0, 1)$ must be contained in a finite union of fibers $\{g\}\times [0,1)$, $g\in P$. On the other hand, an $\cal L$-definable subgroup of $G$ is a topological group containing some $\cal L$-definable neighborhood of $0$ and, thus, also every fiber $\{n\}\times [0,1)$, $n\in \Z$.

The reader can verify that for $\mathcal B=Fin(M)$, $K=P$, $\mathcal U=\mathcal B\times K$ and $\tau(x, y)= (x, y) \mod (1, 1)$, we obtain the diagram of Conjecture \ref{conjecture}. 

Finally, observe that $G$ is a subgroup of the $\cal L$-definable group $B$, which \emph{is} the direct product $B= S\times \la M, +\ra$, where $S$ has domain $\{(x, x): 0\le x<1\}$ and operation $(x, y)\mapsto x+y \mod (1, 1)$.
\end{example}

We finish with some  open questions which we expect our local theorem to have an impact on.

\begin{question}
Does $G$, up to definable isomorphism, contain an $\cal L$-definable local subgroup (in the sense of \cite[\S 23 (D)]{pon}) whose dimension  equals $\ldim G$?
\end{question}
\begin{question}
Assume $\ldim G=\dim cl(G)$. Is $G$, up to definable isomorphism,  $\cal L$-definable?
\end{question}

If Conjecture \ref{conjecture} is true,  it would be nice to know what the small groups are.

\begin{question}
Is every small definable group/set  definably isomorphic to a group/set definable in the induced structure on $P$?
\end{question}

\begin{question}
Is $G$, up to definable isomorphism, a subgroup of an $\cal L$-definable group (whose dimension might be bigger than $\ldim(G)$)?
\end{question}

\end{document}